\documentclass[a4paper,12pt]{article}
\usepackage{amsmath,amsthm,amssymb,latexsym,txfonts}

\usepackage[usenames]{color}
\usepackage[ansinew]{inputenc}

\definecolor{rojo}{rgb}{1,0,0}
\definecolor{blanco}{rgb}{1,1,1}

\newcommand{\bigint}{\begin{picture}(10,10)
\put(-1,2){\line(1,0){10}}
\end{picture}\kern-14pt\int}

\title{{\bf A new characterization of Sobolev spaces on \boldmath$\Rn$ \\
}}
\author{\Large{\Large Roc Alabern, Joan Mateu and Joan Verdera}}
\newtheorem{teo}{Theorem}
\newtheorem{teor}{Theorem}

\newtheorem{lemma}[teo]{Lemma}

\newtheorem*{CZ}{Lemma}

\theoremstyle{definition}
\newtheorem*{gracies}{Acknowledgements}

\newcommand{\Rn}{{\mathbb R}^n}

\newcommand{\R}{\mathbb R}

\begin{document}
\date{}

\maketitle


\begin{abstract}

In this paper we present a new characterization of Sobolev spaces
 on $\Rn$. Our characterizing condition is obtained via a quadratic multiscale expression
 which exploits the particular symmetry properties of Euclidean space. An
interesting feature of our condition is that depends only on the
metric of $\Rn$ and the Lebesgue measure, so that one can define
Sobolev spaces of any order of smoothness on any metric measure
space.
\end{abstract}

\section{Introduction}

In this paper we present a new characterization of the Sobolev
spaces $W^{\alpha,p}$ on $\Rn,$ where the smoothness index $\alpha$
is any positive real number and $ 1 < p < \infty$. Thus
$W^{\alpha,p}$ consists of those functions $f \in L^p =L^p(\Rn)$
such that $(-\Delta)^{\alpha/2} f \in L^p$. Here $\Delta$ is the
Laplacian and $(-\Delta)^{\alpha/2} f$ is defined on the Fourier
transform side by $ |\xi|^\alpha \hat{f}(\xi).$ If $0< \alpha <n$
this means that $f$ is a function in $L^p$ which is the Riesz
potential of order $\alpha$ of some other function~$g$ in $L^p$,
namely $f = c_n 1/|x|^{n-\alpha}\ast g$. If $\alpha$ is integer,
then $W^{\alpha,p}$ is the usual space of those functions in $L^p$
such that all distributional derivatives up to order $\alpha$ are in
$L^p.$

To convey a feeling about the nature of our condition we first
discuss the case $\alpha=1.$  Consider the square function
\begin{equation*}
S(f)^2(x) = \int_0^\infty \left| \frac{f_{B(x,\,t)} -
f(x)}{t}\right|^2\,\frac{dt}{t}, \quad x \in \Rn\,.
\end{equation*}
Here $f$ is a locally integrable function on $\Rn$ and
$f_{B(x,\,t)}$ denotes the mean of $f$ on the open ball with
center~$x$ and radius~$t$. One should think of $\frac{f_{B(x,\,t)} -
f(x)}{t}$ as a quotient of increments of $f$ at the point $x.$ Our
characterization of $W^{1,p}$ reads as follows.

\begin{teor}\label{T1}
If $1 < p <\infty$, then the following are equivalent.

(1) $f \in W^{1,p}$

(2) $f \in L^p$ and $S(f) \in L^p$.

If any of the above conditions holds then
$$
\|S(f)\|_p  \simeq \|\nabla f\|_p\,.
$$
\end{teor}

The symbol $A \simeq B $ means, as usual, that for some constant~$C$
independent of the relevant parameters attached to the quantities
$A$ and $B$ we have $ C^{-1} \,B \leq A \leq C \,B .$

Notice that condition (2) in Theorem \ref{T1} above is of a metric
measure space character, because only involves integrals over balls.
It can be used to define in any metric measure space~$X$ a notion of
Sobolev space $W^{1,p}(X)$. It is not clear to the authors what are
the relations of this space with other known notions of Sobolev
space in a metric measure space, in particular with those of Hajlasz
\cite{H}  or Shanmugalingam \cite{S} (see also \cite{HK}).

The proof of Theorem \ref{T1} follows a classical route (see
\cite{Str}). The relevant issue is the necessary condition. First,
via a Fourier transform estimate we show that
$$
\|S(f)\|_2 = c \,\|\nabla f\|_2\,,
$$
for good functions $f.$  In a second step, we set up a singular
integral operator $T$ with values in $L^2(dt/t)$ such that
$$
\|T(f)\|_{L^2 (\Rn,\, L^2(dt/t))} = \|S(f)\|_2\,.
$$
The kernel of $T$ turns out to satisfy H\"{o}rmander's condition, so
that we can appeal to a well known result of Benedek, Calder\'{o}n and
Panzone  \cite[Theorem 3.4, p.~492]{GR} on vector valued
Calder\'{o}n-Zygmund Theory to conclude the proof. The major technical
difficulty occurs in checking H\"{o}rmander's condition.

The proof extends without pain to cover orders of smoothness
$\alpha$ with $0 < \alpha < 2 .$ The square function $S(f)$ has to
replaced by
$$
S_\alpha(f)^2(x) = \int_0^\infty \left| \frac{f_{B(x,t)} -
f(x)}{t^\alpha}\right|^2\,\frac{dt}{t}, \quad x \in \Rn\,.
$$
The result is then that, for $0 < \alpha < 2 ,$  $f \in
W^{\alpha,p}$ is equivalent to $f \in L^p$ and $S_\alpha(f) \in L^p
.$

Notice that
\begin{equation}\label {eq2}
S_\alpha(f)^2(x) = \int_0^\infty \left| \fint_{B(x ,\, t)}
\frac{f(y)-f(x)}{t^\alpha} \, dy \right|^2 \frac{dt}{t}, \quad x \in
\Rn\,,
\end{equation}
where the barred integral on a set stands for the mean over that
set. Stricharzt (\cite{Str}) used long ago the above square function
for $0 <\alpha< 1$ to characterize $W^{\alpha,p}.$  However the
emphasis in \cite{Str} was on a larger variant of $S_\alpha(f)$ in
which the absolute value is inside the integral in $y$ in
\eqref{eq2}. In the interval $1 \leq \alpha < 2$ putting the
absolute value inside the integral destroys the characterization,
because then one gives up the symmetry properties of $\Rn.$ For
instance, $S_\alpha(f)$ vanishes if $f$ is a first degree
polynomial.

There are in the literature square functions very close to
\eqref{eq2} which characterize  $W^{\alpha,p},$ for $0 < \alpha < 2
$. For example, first differences of $f$ may be replaced by second
differences and the absolute value may be placed inside the integral
(\cite{Str} and \cite [Chapter V]{St}). The drawback with second
differences is that they do not make sense in the setting of metric
measure spaces. See also the paper by Dorronsoro \cite{D}.

We now proceed to explain the idea for the characterization of
$W^{2,p}.$ Take a smooth function~$f$ and consider its Taylor
expansion up to order~$2$ around $x$
\begin{equation}\label{eq3}
 f(y)= f(x)+
\nabla f(x)\cdot (y-x) + \sum_{|\beta|=2} \partial^\beta f(x)
(y-x)^\beta + R \,,
\end{equation}
where $R$ is the remainder and $\beta$ a multi-index of length $2.$
Our goal is to devise a square function which plays the role of
$S_1(f)$ (see \eqref{eq2} for $\alpha=1$) with respect to second
order derivatives. The first remark is that the mean on $B(x,t)$ of
the homogeneous polynomial of degree $1$ in \eqref{eq3} is zero.
Now, the homogeneous Taylor polinomial of degree $2$ can be written
as
\begin{equation}\label{eq4}
\sum_{|\beta|=2} \frac{\partial_\beta f(x)}{\beta !} (y-x)^\beta  =
H(y-x)+ \frac{1}{2n} \Delta f(x)\,|y-x|^2\,,
\end{equation}
for a harmonic homogeneous polynomial $H$ of degree $2$. Hence the
mean on $B(x,t)$ of the homogeneous Taylor polinomial of degree $2$
is
$$
\fint_{B(x,\,t)} \frac{1}{2n} \Delta f(x)\,|y-x|^2 \,dy\,.
$$
This suggests defining
\begin{equation}\label{eq5}
S_2(f)(x)^2 = \int_0^\infty \left| \fint_{B(x,\,t)}
\frac{\left(f(y)-f(x)- \frac{1}{2n} (\Delta f)_{B(x,\,t)}\,|y-x|^2
\right)}{t^2}\,dy \right|^2\,\frac{dt}{t}, \quad x \in \Rn\,.
\end{equation}
We cannot replace $(\Delta f)_{B(x,\,t)}$ by $\Delta f(x)$ in the
preceding definition, because the mean guarantees a little extra
smoothness which one needs in a certain Fourier transform
computation. Notice that, according to the remarks made before on
the mean on the ball $B(x,t)$ of the homogeneous Taylor polynomials
of degrees $1$ and $2$, in the expression above for $S_2(f)(x)$ one
may add the missing terms to get the full Taylor polynomial of
degree $2$, except for the fact that $\Delta f(x)$ should be
replaced by $(\Delta f)_{B(x,\,t)}$. Were $f$ smooth enough, one
could even add the homogeneous Taylor polynomial of degree $3$,
because it is odd (taking $x$ as the origin) and thus its mean on
$B(x,t)$ vanishes. This explains why whatever we can prove for
$\alpha = 2$ will also extend to the range $2 < \alpha < 4$ by
considering an appropriate square function, which turns out to be
\begin{equation*}
S_\alpha (f)(x)^2 = \int_0^\infty \left| \fint_{B(x,\,t)}
\frac{\left(f(y)-f(x)- \frac{1}{2n} \Delta f (x) \,|y-x|^2
\right)}{t^\alpha}\,dy \right|^2\,\frac{dt}{t}, \quad x \in \Rn\,.
\end{equation*}
One should remark that in the range $2 < \alpha < 4$ the mean
$(\Delta f)_{B(x,\,t)}$ can safely be replaced by $\Delta f(x)\,.$
Here is our second order theorem.

\begin{teor}\label{T2}
If $1 < p < \infty$, then the following are equivalent.

(1) $f \in W^{2,p}$

(2) $f \in L^p$ and there exists a function $g \in L^p$  such that
$S_2(f,g) \in L^p,$  where the square function $S_2(f,g)$ is defined
by
$$
S_2(f,g)(x)^2 = \int_0^\infty \left| \fint_{B(x,\,t)}
\frac{\left(f(y)-f(x)- g_{B(x,\,t)}\,|y-x|^2 \right)}{t^2}\,dy
\right|^2\,\frac{dt}{t}, \quad x \in \Rn\,.
$$

If $f \in W^{2,p}$ then one can take $g = \Delta f / 2n$ and if (2)
holds then necessarily $g = \Delta f / 2n$, a. e.

If any of the above conditions holds then

$$
\|S(f,g)\|_p \simeq  \|\Delta f\|_p\,.
$$
\end{teor}

Notice that condition (2) in Theorem \ref{T2} only involves the
Euclidean distance on $\Rn$ and integrals with respect to Lebesgue
measure. Thus one may define a notion of $W^{2,p}(X)$ on any metric
measure space $X$. For more comments on that see section 4.

Again the special symmetry properties of $\Rn$ play a key role. For
instance, $S_2$ annihilates second order polynomials. Theorem
\ref{T2} has a natural counterpart for smoothness indexes $\alpha$
satisfying $ 2 < \alpha < 4.$ The result states that a function $f
\in W^{\alpha,p}$ if and only if $f \in L^p$ and there exists a
function $g \in L^p$ such that $S_\alpha(f,g) \in L^p,$ where
$$
S_\alpha(f,g)(x)^2 = \int_0^\infty \left| \fint_{B(x,\,t)}
\frac{\left(f(y)-f(x)- g(x)\,|y-x|^2 \right)}{t^\alpha}\,dy
\right|^2\,\frac{dt}{t}, \quad x \in \Rn\,.
$$
Notice that in the range  $ 2 < \alpha < 4$ we do not need to
replace $g(x)$ by the mean $g_{B(x,\,t)}.$

Before stating our main result, which covers all orders of
smoothness and all~$p$ with $1 < p < \infty,$ we need to discuss a
couple of preliminary issues. The first is  the analogue of
\eqref{eq4} for homogeneous polynomials of any even degree. Let $P$
be a homogeneous polynomial of degree~$2j.$ Then $P$ can be written
as
$$
P(x)= H(x) + \Delta^j P \;\frac{1}{L_j}\,|x\,|^{2j}\,,
$$
where $L_j = \Delta^j (|x\,|^{2j})$ and $H$ satisfies $\Delta^j H =
0.$ This follows readily from \cite[3.1.2, p.~69]{St}. Considering
the spherical harmonics expansion of $P(x)$ we see that
$\int_{|x|=1} H(x)\,d\sigma = 0$, $\sigma$~being the surface measure
on the unit sphere, and thus that $\int_{|x|\leq t} H(x)\,dx = 0$,
$t > 0.$ The precise value of $L_j,$ which can be computed easily,
will not be needed.

Our main result involves a square function associated with a
positive smoothness index $\alpha.$ Let $N$ be the unique integer
such that $2N \le \alpha < 2N+2\,.$ Given locally integrable
functions $f, g_1, \dots , g_N$ we set
$$
S_\alpha(f,g_1,g_2,\dotsc,g_N)(x)^2 = \int_0^\infty \left|
\fint_{B(x,\,t)} \frac{R_N(y,x)}{t^\alpha}\,dy
\right|^2\,\frac{dt}{t}, \quad x \in \Rn\,,
$$
where
$$
R_N(y,x)= f(y)-f(x)- g_1(x)\,|y-x|^2 -\dotsb-
g_{N-1}(x)\,|y-x|^{2(N-1)}- (g_N)_{B(x,\,t)}|y-x|^{2N}
$$
if $\alpha = 2N \,,$ and
$$
R_N(y,x)= f(y)-f(x)- g_1(x)\,|y-x|^2 -\dotsb-
g_{N-1}(x)\,|y-x|^{2(N-1)}- g_N (x)\,|y-x|^{2N}
$$
if $2N <\alpha <  2N+2 \,.$

\begin{teor}\label{T3}
Given $\alpha > 0$ let $N$ be the unique integer such that $2N \leq
\alpha  < 2N+2.$  If $1 < p < \infty$, then the following are
equivalent.

(1) $f \in W^{\alpha,p}$

(2) $f \!\in\! L^p$ and there exist functions $g_j \!\in\! L^p$,
$\!1\! \!\leq\! j \!\leq\! N$, such that
$S_\alpha(f,g_1,g_2,\dotsc,g_N) \!\in\! L^p$.

 If  $f \in W^{\alpha,p}$ then one can take $g_j = \Delta^j f /
L_j$ and if (2) holds then necessarily $g_j = \Delta^j f / L_j$\,
a. e.

If any of the above conditions holds then
$$
\|S_\alpha (f,g_1,\dotsc,g_N)\|_p \simeq \| (-\Delta)^{\alpha/2}
f\|_p\,.
$$
\end{teor}

Again condition (2) in Theorem \ref{T3} only involves the Euclidean
distance on $\Rn$ and integrals with respect to Lebesgue measure.
Thus one may define a notion of $W^{\alpha,p}(X)$ for any positive
$\alpha$ and any $1 <  p < \infty$ on any metric measure space~$X$.
For previous notions of higher order Sobolev spaces on metric
measure spaces see \cite{LLW}. See section 4 for more on that.

The proof of Theorem \ref{T3} proceeds along the lines sketched
before for $\alpha =1.$ First we use a Fourier transform computation
to obtain the relation
$$
\|S_{\alpha}(f, \Delta f / L_1,\dotsc, \Delta^N f / L_N)\|_2 = c \,
\| (-\Delta)^{\alpha/2} f\|_2\,.
$$
Then we introduce a singular integral operator with values in
$L^2(dt/t^{2\alpha+1})$ and we check that its kernel satisfies
H\"{o}rmander's condition.

The paper is organized as follows. In sections 1, 2 and 3  we prove
respectively Theorems \ref{T1}, \ref{T2} and \ref{T3}. In this way
readers interested only in first order Sobolev spaces may
concentrate in section 1. Those readers interested in the main idea
about jumping to orders of smoothness $2$ and higher may read
section 2. Section 3 is reserved to those interested in the full
result. In any case the technical details for the proof of Theorem
\ref{T1} are somehow different of those for orders of smoothness $2$
and higher. The reason is that H\"{o}rmander's condition involves
essentially taking one derivative of the kernel and is precisely the
kernel associated to the first order of smoothness that has minimal
differentiability.

Our notation and terminology are standard. For instance, we shall
adopt the usual convention of denoting by $C$ a constant independent
of the relevant variables under consideration and not necessarily
the same at each occurrence.

If $f$ has derivatives of order $M$ for some non-negative integer
$M$, then $\nabla^M f \!=\! (\partial^\beta f)_{|\beta|=M}$ is the
vector with components the partial derivatives of order $M$ of $f$
and $|\nabla^M f|$ its Euclidean norm.

 The Zygmund class on $\Rn$ consists of those continuous
functions $f$ such that, for some constant $C$,
$$
|f(x+h)+f(x-h)-2 f(x)| \le C\,|h|,\quad x, h \in \Rn\,.
$$
The basic example of a function in the Zygmund class which is not
Lipschitz is $f(x)= |x| \log |x|$, $x \in \Rn.$

The Schwartz class consists of those infinitely differentiable
functions on $\Rn$ whose partial derivatives of any order decrease
faster than any polynomial at $\infty$.

After the first draft of the paper was made public Professor Wheeden
brought to our attention his articles \cite{W1} and \cite{W2}. In
\cite{W1} a general result is proven which, in particular, contains
Theorem 3 for $0<  \alpha < 2 $. In \cite{W2} Sobolev spaces with
respect to special homogeneous spaces are considered for $0 < \alpha
< 1 \,.$

\section{Proof of Theorem \ref{T1}}

The difficult part is the necessity of condition (2) and we start
with this.

As a first step we show that
\begin{equation}\label{eq7}
\|S_1(f)\|_2 = c \,\|\nabla f\|_2
\end{equation}
for a dimensional constant $c.$  Set
$$\chi(x)= \frac{1}{|B(0,1)|}\, \chi_{B(0,1)}(x)$$
and
$$
\chi_t(x)= \frac{1}{t^n} \chi(\frac{x}{t})\,,
$$
so that, by Plancherel,
\begin{equation*}
\begin{split}
\int_{\Rn}  S_1 (f)(x)^2 \,dx & = \int_0^\infty \int_{\Rn} \left|(f
\ast \chi_t)(x) - f(x)\right|^2 \,dx \;\frac{dt}{t^3} \\*[5pt]
 & = c\, \int_0^\infty \int_{\Rn} \left|\hat{\chi}(t \xi)-1\right|^2 \left|\hat{f}(\xi) \right|^2\,d\xi \;
 \frac{dt}{t^3}\,.
\end{split}
\end{equation*}
Since $\hat{\chi}$ is radial, $\hat{\chi}(\xi) = F(|\xi|)$ for a
certain function $F$ defined on $[0,\infty).$  Exchange the
integration in $d\xi$ and $dt$ in the last integral above and make
the change of variables $\tau = t\, |\xi|.$ Then
\begin{equation*}
\begin{split}
\int_{\Rn}  S_1 (f)(x)^2 \,dx & = c \,\int_{\Rn} \int_0^\infty
\left|F(\tau)-1 \right|^2 \,\frac{d\tau}{\tau^3}\; |\hat{f}(\xi)|^2
|\xi|^2 \,d\xi \\*[5pt]
 & = c\,\int_0^\infty  \left|F(\tau)-1 \right|^2 \,\frac{d\tau}{\tau^3} \;\; \|\nabla f\|_2^2
\end{split}
\end{equation*}
and \eqref{eq7} is reduced to showing that
\begin{equation}\label{eq8}
\int_0^\infty \left|F(\tau)-1 \right|^2 \,\frac{d\tau}{\tau^3} <
\infty \,.
\end{equation}
Set $B = B(0,1)$ and $e_1 =(1,0,\dotsc,0) \in \Rn.$ Then
\begin{equation*}
\begin{split}
F(t) & = \hat{\chi}(t e_1) = \fint_{B} \exp{(- i x_1 t)}
\,dx\\*[5pt] & = \fint_{B} \left(1-i x_1 t - \frac{1}{2} x_1^2
t^2+\dotsb \right)\,dx \\*[5pt] &= 1- \frac{1}{2} \,\fint_{B} x_1^2
\,dx \;t^2+\dotsb,
\end{split}
\end{equation*}
which yields
$$
F(t)-1 = O(t^2), \quad \text{as}\quad t \rightarrow 0
$$
and shows the convergence of \eqref{eq8} at $0.$

Since $F(|\xi|) = \hat{\chi}(\xi)$ is the Fourier transform of an
integrable function, $F(\tau)$ is a bounded function and so the
integral \eqref{eq8} is clearly convergent at $\infty.$


We are left with the case of a general $p$ between $1$ and $\infty.$
If $f \in W^{1,p},$ then $f = g \ast 1/|x|^{n-1}$ for some $g \in
L^p$ (with $1/|x|^{n-1}$ replaced by $\log|x|$ for $n=1$). Set
$I(x)= 1/|x|^{n-1}.$ Then
\begin{equation*}
f_{B(x,\,t)}-f(x)  = (f \ast \chi_t)(x) -f(x)  = (g \ast K_t)(x)\,,
\end{equation*}
where
\begin{equation}\label{eq9}
K_t(x)= (I \ast \chi_t)(x)-I(x)= \fint_{B(x,\,t)} I(y)\,dy -I(x)\,.
\end{equation}
If we let $T(g)(x)= (g \ast K_t)(x),\; x \in \Rn$, then  one can
rewrite $S_1 (f)(x)$ as
$$
S_1 (f)(x) = \left(\int_0^\infty \left| (g \ast K_t)(x)
\right|^2\,\frac{dt}{t^3}\right)^{\frac{1}{2}} = \|Tg(x)
\|_{L^2(dt/t^3)}\,.$$ Then \eqref{eq7} translates into
$$
\int_{\Rn} \|Tg(x) \|_{L^2(dt/t^3)}^2 \,dx = c\,\|g\|_2^2\,,
$$
and we conclude that $T$ is an operator mapping isometrically
$L^2(\Rn)$ into $L^2(\Rn, L^2(dt/t^3)).$ If the kernel $K_t(x)$ of
$T$ satisfies H\"{o}rmander's condition
$$
\int_{|x|\geq 2 |y|} \|K_t(x-y)-K_t(x)\|_{L^2(dt/t^3)}\,dx \le
C,\quad y \in \Rn
$$
then a well known result of Benedek, Calder\'{o}n and Panzone on vector
valued singular integrals (see \cite[Theorem 3.4, p.~492]{GR})
yields the $L^p$ estimate
$$
\int_{\Rn} \|Tg(x) \|_{L^2(dt/t^3)}^p \,dx \leq C_p\,\|g\|_p^p \,,
$$
which can be rewritten as
$$
\|S_1(f)\|_p \le C_p \,\|\nabla f \|_p \,.
$$
The reverse inequality follows from polarization from \eqref{eq7}
 by a well known argument (\cite[p.~507]{GR}) and so the proof of the necessary
condition is complete.
 We are going to
prove the following stronger version of H\"{o}rmander's condition
\begin{equation}\label{eq10}
\|K_t(x-y)-K_t(x)\|_{L^2(dt/t^3)} \le C\,\frac{|y|}{|x|^{n+1}},\quad
y \in \Rn\,,
\end{equation}
for almost all $x$ satisfying $|x|\geq 2 |y|.$

To prove \eqref{eq10} we deal separately with three intervals in the
variable $t.$

\vspace*{7pt}

Interval 1: $t < \frac{|x|}{3}.$ From the definition of $K_t$ in
\eqref{eq9} we obtain
\begin{equation*}
 \nabla K_t(x)= (\nabla I \ast \chi_t)(x)- \nabla I(x)\,.
\end{equation*}
Notice that, in the distributions sense, the gradient of $I$ is a
constant times the vector valued Riesz transform, namely
$$
\nabla I = -(n-1) p.v. \frac{x}{|x|^{n+1}}\,.
$$
If $|x|\geq 2 |y|$, then the segment $[x-y,x]$ does not intersect
the ball $B(0,|x\,|/2)$ and thus
\begin{equation}\label{eq12}
|K_t(x-y)-K_t(x)| \le |y| \sup_{z \in [x-y,y]}|\nabla K_t(z)| \,.
\end{equation}
If $t < |x|/3$ and $z \in [x-y,y]$, then $B(z,t) \subset \Rn
\setminus B(0,|x \,|/6),$ and hence
\begin{equation}\label{eq13}
\nabla K_t(z)= \fint_{B(z,\,t)} (\nabla I(w)- \nabla I(z))\,dw \,.
\end{equation}
Taylor's formula up to order $2$ for $\nabla I (w)$ around $z$
yields
$$
\nabla I (w) = \nabla I (z)+ \nabla^2
I(z)(w-z)+O(\frac{|w-z|^2}{|x|^{n+2}})\,,
$$
where $\nabla^2 I(z)(w-z)$ is the result of applying the matrix
$\nabla^2 I(z)$ to the vector $w-z.$ The mean value of $\nabla^2
I(z)(w-z)$ on $B(z,t)$ is zero, by antisymmetry, and thus, by
\eqref{eq13},
$$
|\nabla K_t(z)| \le C\, \frac{t^2}{|x|^{n+2}}
$$
and so, by \eqref{eq12}
$$
|K_t(x-y)-K_t(x)| \le C\, |y|\frac{t^2}{|x|^{n+2}}\,.
$$
Integrating in $t$ we finally get
\begin{equation*}
\left(\int_0^{|x|/3} |K_t(x-y)-K_t(x)|^2
\,\frac{dt}{t^3}\right)^{\frac{1}{2}} \le C \,\frac{|y|}{|x|^{n+2}}
\left(\int_0^{|x|/3} t\,dt \right)^{\frac{1}{2}} =
C\,\frac{|y|}{|x|^{n+1}}\,.
\end{equation*}

\vspace*{7pt}

Interval 2:  $|x\,|/3 < t <  2 |x\,|.$ The function $I \ast \chi_t$
is continuously differentiable on $\Rn \setminus S_t$, $S_t = \{x :
|x|=t\},$ because its distributional gradient is given by $I \ast
\nabla \chi_t$ and each component of $\nabla \chi_t$ is a Radon
measure supported on $S_t.$ The gradient of $I \ast \chi_t$ is given
at each point $x \in \Rn \setminus S_t$ by the principal value
integral
$$
p.v. (\nabla I \ast \chi_t)(x)= -(n-1) p.v. \fint_{B(x,\,t)}
\frac{y}{|y\,|^{n+1}}\,,
$$
which exists for all such $x.$ The difficulty in the interval under
consideration is that it may happen that $|x|=t$ and then the
gradient of $I \ast \chi_t$ has a singularity at such an $x.$ We
need the following estimate.

\begin{lemma}\label{PV}
$$\left| p.v. \int_{B(x,\,t)} \frac{y}{|y|^{n+1}}\,dy \right| \le C\,
\log \frac{|x\,|+t}{||x\,|-t|},\quad x \in \Rn\,.
$$
\end{lemma}

\begin{proof}
Assume without loss of generality that $x=(x_1,0,\dotsc,0).$ The
coordinates $y_j$, $j \neq 1$, change sign under reflection around
the $y_1$ axes. Hence
$$
p.v. \int_{B(x,\,t)} \frac{y_j}{|y|^{n+1}}\,dy = 0, \quad 1 <j \leq
n\,.
$$
Now, if $|x\,|< t,$
\begin{equation*}
\begin{split}
\left| p.v. \int_{B(x,\,t)} \frac{y_1}{|y|^{n+1}}\,dy \right| & =
\left| p.v. \int_{B(x,\,t)\setminus B(0,\,t-|x|)}
\frac{y_1}{|y|^{n+1}}\,dy \right|\\*[5pt] & \le C\,
\int_{t-|x|}^{t+|x|} \frac{dt}{t} = C\,\log \frac{t+|x|}{t-|x|}\,.
\end{split}
\end{equation*}

If $|x\,|> t,$
\begin{equation*}
\begin{split}
\left| p.v. \int_{B(x,\,t)} \frac{y_1}{|y|^{n+1}}\,dy \right| & =
\left|\int_{B(x,\,t)} \frac{y_1}{|y|^{n+1}}\,dy \right|\\*[5pt] &
\le C\, \int_{|x|-t}^{|x|+t} \frac{dt}{t} = C\,\log
\frac{|x|+t}{|x|-t}\,.
\end{split}
\end{equation*}
\end{proof}

Assume without loss of generality that $y =(y_1,0,\dotsc,0).$ The
distributional gradient of~$I \ast \chi_t$ is
$$
 -(n-1) p.v.
\frac{y}{|y|^{n+1}} \ast \chi_t \,,
$$
which is in $L^2.$ Then $I \ast \chi_t \in W^{1,2}$ and consequently
is absolutely continuous on almost all lines parallel to the first
axes. Therefore
$$
K_t(x-y)-K_t(x) = - \int_0^1 \nabla K_t(x- \tau y) \cdot y \,d\tau
$$
for almost all $x$ and
$$
|K_t(x-y)-K_t(x)| \le C\, \frac{|y|}{|x|^{n}}\,\int_0^1 \left(1+\log
\frac{|x- \tau y|+t}{||x- \tau y|-t|}\right)\,d\tau\,.
$$
Hence
\begin{equation*}
\begin{split}
\left(\int_{|x|/3}^{2|x|} |K_t(x-y)-K_t(x)|^2
\,\frac{dt}{t^3}\right)^{\frac{1}{2}} & \le C
\,\frac{|y|}{|x|^{n+1}} \left(\int_{|x|/3}^{2|x|} \left( \int_0^1
\left(1+\log \frac{|x- \tau y|+t}{||x- \tau y|-t|}\right) \,d\tau
\right)^2 \frac{dt}{t} \right)^{\frac{1}{2}} \\*[5pt] & =
C\,\frac{|y|}{|x|^{n+1}}\, D\,,
\end{split}
\end{equation*}
where the last identity is a definition of $D.$ Applying Schwarz to
the inner integral in $D$ and then changing the order of integration
we get
$$
D^2 \le\int_0^1 \left( \int_{|x|/3}^{2|x|}\left( 1+ \log \frac{|x-
\tau y|+t}{||x- \tau y|-t|}\right)^2 \,\frac{dt}{t} \right) \,d\tau
\,.
$$
For each $\tau$ make the change of variables
$$
s = \frac{t}{|x- \tau y|}
$$
to conclude that
$$
D^2 \le \int_{2/9}^4 \left(1+\log
\frac{1+s}{|1-s|}\right)^2\,\frac{ds}{s}\,.
$$

\vspace*{7pt}

Interval 3: $ 2 |x\,| \leq t .$  For each $z$ in the segment
$[x-y,y]$ we have $B(0,t/4) \subset B(z,t).$ Then, by~\eqref{eq13},
\begin{equation*}
\begin{split}
\nabla K_t(z) & = -(n-1) \left(p.v. \frac{1}{|B(z,t)|}
\int_{B(z,\,t)} \frac{w}{|w|^{n+1}}\,dw -
\frac{z}{|z|^{n+1}}\right)\\*[5pt] & = -(n-1) \left(
\frac{1}{|B(z,t)|} \int_{B(z,\,t)\setminus B(0,\,t/4)}
\frac{w}{|w|^{n+1}}\,dw - \frac{z}{|z|^{n+1}}\right)
\end{split}
\end{equation*}
and so
$$
|\nabla K_t(z)| \le C\, \frac{1}{|x|^n},\quad z \in [x-y,y]\,.
$$
Hence, owing to \eqref{eq12},
$$
|K_t(x-y)-K_t(x)| \le C\, \frac{|y|}{|x|^{n}}
$$
and thus
$$
\left(\int_{2|x |}^{\infty} |K_t(x-y)-K_t(x)|^2
\,\frac{dt}{t^3}\right)^{\frac{1}{2}} \le C \,\frac{|y|}{|x|^{n}}
\left(\int_{2|x|}^{\infty} \frac{dt}{t^3} \right)^{\frac{1}{2}} =
C\,\frac{|y|}{|x|^{n+1}}\,,
$$
which completes the proof of the strengthened form of H\"{o}rmander's
condition \eqref{eq10}.

We turn now to prove that condition (2) in Theorem \ref{T1} is
sufficient for $f \in W^{1,p}.$ Let $f \in L^p$ satisfy $S_1(f) \in
L^p. $ Take an infinitely differentiable function $\phi \geq 0$ with
compact support in $B(0,1),$ $\int \phi = 1$ and set
$\phi_\epsilon(x)= \frac{1}{\epsilon^n} \phi(\frac{x}{\epsilon})$,
$\epsilon>0.$ Consider the regularized functions $f_\epsilon = f
\ast \phi_\epsilon.$ Then $f_\epsilon$ is infinitely differentiable
and $\|\nabla f_\epsilon\|_p \le \|f\|_p \|\nabla
\phi_\epsilon\|_1,$ so that $f_\epsilon \in W^{1,p}.$ Thus, as we
have shown before,
$$
\|\nabla f_\epsilon\| \simeq \|S_1(f_\epsilon)\|_p\,.
$$
We want now to estimate $\|S_1(f_\epsilon)\|_p$ independently of
$\epsilon.$  Since
$$
(f_\epsilon)_{B(x,\,t)} - f_\epsilon(x) = \left((f \ast \chi_t
-f)\ast\phi_\epsilon \right)(x)\,,
$$
Minkowsky's integral inequality gives
\begin{equation*}
S_1(f_\epsilon)(x)  = \|(f_\epsilon)_{B(x,t)} -
f_\epsilon(x)\|_{L^2(dt/t^3)} \le (S_1(f)\ast \phi_\epsilon)(x)\,,
\end{equation*}
and so $\|\nabla f_\epsilon\| \le C\,\|S_1(f)\|_p$, $\epsilon
>0. $ For an appropriate sequence $\epsilon_j \rightarrow 0$ the
sequences~$ \partial_k f_{\epsilon_j}$ tend in the weak $\star$
topology of $L^p$ to some function $g_k \in L^p$, $1 \le k \le n .$
On the other hand, $f_\epsilon \rightarrow f$ in $L^p$ as $\epsilon
\rightarrow 0$ and thus $\partial_k f_\epsilon \rightarrow
\partial_k f $, $1 \le k \le  n$ in the weak  topology of distributions.
Therefore $\partial_k f = g_k$ for all $k$ and so $f \in W^{1,p}.$

\section{Proof of Theorem \ref{T2}}

The difficult direction is (1) implies (2) and this is the first we
tackle. We start by showing that if $f \in W^{2,2}$ then
\begin{equation}\label{eq15}
\|S_2(f)\|_2 = c\, \|\Delta f\|_2
\end{equation}
 where the square function~$S_2(f)$
is defined in \eqref{eq5}. To apply Plancherel in the $x$
 variable it is convenient to write the innermost integrand in \eqref{eq5} as
\begin{equation*}
\begin{split}
& \fint_{B(x,\,t)} \left(f(y)-f(x)- \left(\fint_{B(x,\,t)}
\frac{\Delta f(z)}{2n}\,dz \right) |y-x|^2\right)\,dy \\*[5pt] =&
\fint_{B(0,\,t)} \left(f(x+h)-f(x)- \left(\fint_{B(0,\,t)}
\frac{\Delta f(x+k)}{2n}\,dk \right) |h|^2\right)\,dh\,.
\end{split}
\end{equation*}
Applying Plancherel we get, for some dimensional constant~$c$,
$$
c\,\|S_2(f)\|_2^2 \!=\! \int_0^\infty\int_{\mathbb{R}^n}
\fint_{B(0,\,t)} \left( \exp{(i \xi h)} -1 + \left(\fint_{B(0,\,t)}
\exp{(i \xi k)} \,dk \right) \frac{|h|^2 |\xi|^2}{2n}\right)\,dh\,
|\hat{f}(\xi)|^2\,d\xi \,\frac{dt}{t^5}\,.
$$
Make appropriate dilations in the integrals with respect to the
variables $h$ and $k$ to bring the integrals on $B(0,1).$ Then use
that the Fourier transform of $\frac{1}{|B(0,1)|} \chi_{B(0,1)}$ is
a radial function, and thus of the form $F(|\xi|)$ for a certain
function $F$ defined on $[0,\infty).$ The result is
$$
c\,\|S_2(f)\|_2^2 =\int_{\mathbb{R}^n} \int_0^\infty \left|
F(t\,|\xi|)-1 + t^2 |\xi|^2 F(t\,|\xi|) \frac{1}{2n}
\fint_{B(0,1)}|h|^2\,dh   \right|^2
\,\frac{dt}{t^5}|\hat{f}(\xi)|^2\,d\xi \,.
$$
The change of variables $\tau = t\, |\xi|$ yields
$$
c\,\|S_2(f)\|_2^2 = I \; \|\Delta f\|_2^2
$$
where $I$ is the integral
\begin{equation}\label{eq16}
I = \int_0^\infty \left|F(\tau)-1 + \tau^2 F(\tau) \frac{1}{2n}
\fint_{B(0,1)}|h|^2\,dh \right|^2\,\frac{d\tau}{\tau^5}\,.
\end{equation}
The only task left is to prove that the above integral is finite.
Now, as $\tau \rightarrow 0$,
\begin{equation*}
\begin{split}
F(\tau) & = \fint_{B(0,1)} \exp{(i h_1 \tau)}\,dh \\*[5pt] & =
\fint_{B(0,1)} \left( 1+ i h_1 \tau -\frac{1}{2} h_1^2 \tau^2
+\dotsb \right)\,dh \\*[5pt] & = 1-\frac{1}{2} \left(\fint_{B(0,1)}
h_1^2 \,dh\right) \tau^2 + O(\tau^4)\,.
\end{split}
\end{equation*}
Hence
\begin{multline*}
 F(\tau)-1+ \tau^2 F(\tau) \frac{1}{2n} \fint_{B(0,1)}|h|^2 \,dh
\\*[5pt]
= \left(- \frac{1}{2} \fint_{B(0,1)} h_1^2 \,dh+ \frac{1}{2n}
\fint_{B(0,1)}|h|^2\,dh \right) \tau^2 + O(\tau^4)= O(\tau^4)\,,
\end{multline*}
because clearly $\fint_{B(0,1)}|h|^2\,dh = n \fint_{B(0,1)} h_1^2
\,dh .$ Therefore the integral \eqref{eq16} is convergent
at~$\tau=0.$

To deal with the case $\tau \rightarrow \infty$ we recall that $F$
can be expressed in terms of Bessel functions. Concretely, one has
(\cite[Appendix B.5, p.~429]{Gr})
$$
|B(0,1)|\; F(\tau) = \frac{J_{n/2}(\tau)}{|\tau|^{n/2}}\,.
$$
The asymptotic behavior of $J_{n/2}(\tau)$ gives the inequality
$$
|F(\tau)| \le C\, \frac{1}{\tau^{\frac{n+1}{2}}}\le C\,
\frac{1}{\tau}\,,
$$
which shows that the integral \eqref{eq16} is convergent at
$\infty.$

We turn our attention to the case $1 < p < \infty.$ Let $I_2(x)$
stand for the kernel defined on the Fourier transform side by
$$
\hat{I_2}(\xi)= \frac{1}{|\xi|^2}\,.
$$
In other words, $I_2$ is minus the standard fundamental solution of
the Laplacian.  Thus $I_2(x) = c_n \,1/|x|^{n-2}$ if $n \geq 3$,
$I_2(x) = - \frac{1}{2 \pi}\,\log|x|$ if $n=2$ and $ I_2(x)= -
\frac{1}{2}\,|x|$ if $n=1.$ Given any $f \in W^{2,p}$ there exists
$g \in L^p$ such that $f = I_2 \ast g$ (indeed, $g= - \Delta f$). We
claim that there exists a singular integral operator $T(g)$ taking
values in $L^2(dt/t^5)$ such that
\begin{equation}\label{eq17}
S_2(f)(x)= \|T(g)(x)\|_{L^2(dt/t^5)}\,.
\end{equation}
Set
$$\chi(x)= \frac{1}{|B(0,1)|}\, \chi_{B(0,1)}(x)$$
and
$$
\chi_t(x)= \frac{1}{t^n} \chi(\frac{x}{t})\,.
$$
Then, letting $M = \fint_{B(0,\,1)} |z|^2 \,dz ,$
\begin{equation*}
\begin{split}
\fint_{B(x,\,t)} \left(f(y)-f(x)- \frac{1}{2n} (\Delta
f)_{B(x,\,t)}\,|y-x|^2 \right)\,dy & = ((I_2\ast\chi_t - I_2-
\frac{M}{2n}\,t^2\,\chi_t)\ast g )(x)\\*[5pt] & = (K_t \ast g)(x)\,,
\end{split}
\end{equation*}
where
$$
K_t(x)= (I_2\ast\chi_t)(x) - I_2(x)- \frac{M}{2n}\,t^2\,\chi_t(x)\,.
$$
 Setting $T(g)(x) = (K_t \ast g)(x)$ we get \eqref{eq17} from the
definition of $S_2(f)$ in \eqref{eq5}. Then \eqref{eq15} translates
into
$$
\int_{\Rn} \|Tg(x) \|_{L^2(dt/t^5)}^2 \,dx = c\,\|g\|_2^2 \,,
$$
and we conclude that $T$ is an operator mapping isometrically
$L^2(\Rn)$ into $L^2(\Rn, L^2(dt/t^5))$, modulo the constant $c\,.$
If the kernel $K_t(x)$ of $T$ satisfies H\"{o}rmander's condition
$$
\int_{|x|\geq 2 |y|} \|K_t(x-y)-K_t(x)\|_{L^2(dt/t^5)}\,dx \le
C,\quad y \in \Rn,
$$
then a well known result of Benedek, Calder\'{o}n and Panzone on vector
valued singular integrals (see \cite[Theorem 3.4, p.~492]{GR})
yields the $L^p$ estimate
$$
\int_{\Rn} \|Tg(x) \|_{L^2(dt/t^5)}^p \,dx \leq C_p\,\|g\|_p^p \,,
$$
which can be rewritten as
$$
\|S_2(f)\|_p \le C_p \,\|\Delta f \|_p \,.
$$
The reverse inequality follows from polarization from \eqref{eq15}
by a well known duality argument (\cite[p.~507]{GR})  and so the
proof of the necessary condition is complete.

We are going to prove the following stronger version of H\"{o}rmander's
condition
\begin{equation*}
\|K_t(x-y)-K_t(x)\|_{L^2(dt/t^5)} \le
C\,\frac{|y|^{1/2}}{|x|^{n+1/2}},\quad |x|\geq 2 |y|\,.
\end{equation*}
 For this we deal separately with the kernels $H_t(x)=
(I_2\ast\chi_t)(x) - I_2(x)$ and $t^2\,\chi_t(x).$  For
$t^2\,\chi_t(x)$ we first remark that the quantity
$|\chi_t(x-y)-\chi_t(x)|$ is non-zero only if $|x-y|< t< |x|$ or
$|x|< t< |x-y|,$ in which cases takes the value $1/c_n \,t^n$, $c_n
= |B(0,1)|.$ On the other hand, if $|x|\geq 2 |y|\,$ then each $z$
in the segment joining $x$ and $x-y$ satisfies $|z|\geq |x|/2.$
Assume that $|x-y|<  |x|$ (the case~$|x|< |x-y|$ is similar). Then
\begin{equation*}
\begin{split}
\left(\int_{0}^\infty (t^2 \,(\chi_t(x-y)-\chi_t(x)))^2
\,\frac{dt}{t^5} \right)^{\frac{1}{2}} & =
c\,\left(\int_{|x-y|}^{|x|}
\frac{dt}{t^{2n+1}}\right)^{\frac{1}{2}}\\*[5pt]
 & = c\,\left(
\frac{1}{|x-y|^{2n}}-\frac{1}{|x|^{2n}}\right)^{\frac{1}{2}} \le C\,
\frac{|y|^{1/2}}{|x|^{n+1/2}}\,.
\end{split}
\end{equation*}

We check now that $H_t$ satisfies the stronger form of H\"{o}rmander's
condition. If $t < |x\,|/2$, then the origin does not belong to the
ball $B(x-y,t)$ nor to the ball $B(x,t).$ Since $I_2$ is harmonic
off the origin, the mean of $I_2$ on these balls is the value of
$I_2$ at the center. Therefore $H_t(x-y)-H_t(x)= 0$ in this case.

If $t \geq |x|/2$, then
$$
|H_t(x-y)-H_t(x)| \le |y|\,\sup_{z \,\in \,[x-y,x]} |\nabla H_t(z)|
\le C\, \frac{|y|}{|x|^{n-1}}\,.
$$
The last inequality follows from
$$
\nabla H_t(z) = \fint_{B(z,\,t)} \nabla I_2(w)\,dw - \nabla
I_2(z)\,,
$$
$ |\nabla I_2(z)| \le C\, 1/|z|^{n-1} \le C\, 1/|x|^{n-1}$ and
$$
|\fint_{B(z,\,t)} \nabla I_2(w)\,dw| \le \fint_{B(z,\,t)}
\frac{1}{|w|^{n-1}} \le C\, \frac{1}{|z|^{n-1}}\,.
$$
Therefore
$$
\left(\int_{0}^\infty |H_t(x-y)-H_t(x)|^2 \,\frac{dt}{t^5}
\right)^{\frac{1}{2}} \le C\, \frac{|y|}{|x|^{n-1}}\,
\left(\int_{|x|/2}^\infty \frac{dt}{t^5} \right)^{\frac{1}{2}} =
C\,\frac{|y|}{|x|^{n+1}}\,.
$$

We turn now to prove that condition (2) in Theorem \ref{T2} is
sufficient for $f \in W^{2,p}.$ Let $f$ and $g$ in $L^p$ satisfy
$S_2(f,g) \in L^p. $ Take an infinitely differentiable function
$\phi \geq 0$ with compact support in $B(0,1),$ $\int \phi = 1$ and
set $\phi_\epsilon(x)= \frac{1}{\epsilon^n}
\phi(\frac{x}{\epsilon}),\; \epsilon>0.$ Consider the regularized
functions $f_\epsilon = f \ast \phi_\epsilon$ and  $g_\epsilon =
g\ast \phi_\epsilon.$  Then $f_\epsilon$ is infinitely
differentiable and $\|\Delta f_\epsilon\|_p \le \|f\|_p \,\|\Delta
\phi_\epsilon\|_1,$ so that $f_\epsilon \in W^{2,p}.$  Recalling
that $M= \fint_{B(0,1)} |z|^2 \,dz,$ we get, by Minkowsky's integral
inequality,
\begin{equation*}
\begin{split}
S_2(f_\epsilon,g_\epsilon)(x) & = \|(f_\epsilon \ast
\chi_t)(x)-f_\epsilon(x) -(g_\epsilon \ast \chi_t)(x)\,M^2\,
t^2\|_{L^2(dt/t^5)} \\& = \|\left((f \ast \chi_t)-f -(g \ast
\chi_t)\,M^2\, t^2 \ast \phi_\epsilon \right)(x) \|_{L^2(dt/t^5)}\\&
\le \left(S_2(f,g)\ast \phi_\epsilon \right)(x)\,.
\end{split}
\end{equation*}
Now we want to compare $(1/2n) \Delta f_\epsilon$ and $g_\epsilon.$
Define
$$
D_\epsilon(x)  = \left(\int_0^\infty M^2 \left|\frac{1}{2n} (\Delta
f_\epsilon \ast \chi_t)(x)- (g_\epsilon \ast \chi_t)(x)
\right|^2\,\frac{dt}{t} \right)^{1/2}.
$$
Then
\begin{equation*}
\begin{split}
D_\epsilon(x) & \le S_2(f_\epsilon)(x) + S_2(f_\epsilon,
g_\epsilon)(x)
\\*[5pt] &\le S_2(f_\epsilon)(x) + \left(S_2(f, g)(x)\ast
\phi_\epsilon \right)(x)\,,
\end{split}
\end{equation*}
and thus $D_\epsilon(x)$ is an $L^p$ function. In particular
$D_\epsilon(x) < \infty,$ for almost all $x \in \Rn.$ Hence
$$
|(1/2n) \Delta f_\epsilon (x)- g_\epsilon (x)| = \lim_{t \rightarrow
0} \left| (1/2n) (\Delta f_\epsilon \ast \chi_t)(x)- (g_\epsilon
\ast \chi_t)(x) \right| =0\,,
$$
for almost all $x \in \Rn,$ and so $ (1/2n) \Delta f_\epsilon
\rightarrow g $ in $L^p$ as $\epsilon \rightarrow 0.$ Since
$f_\epsilon \rightarrow f$ in $L^p$ as $\epsilon \rightarrow 0,$
then $\Delta f_\epsilon \rightarrow \Delta f$ in the weak topology
of distributions. Therefore $ (1/2n) \Delta f = g $ and the proof is
complete.

\section{Proof of Theorem \ref{T3}}

The difficult direction in Theorem \ref{T3} is to show that
condition (2)  is necessary for $f \in W^{\alpha,p}.$  The proof
follows the pattern already described in the preceding sections. One
introduces an operator $T$ taking values in $L^2(dt/t^{2\alpha+1})$
and shows via a Fourier transform estimate that $T$ sends $L^2(\Rn)$
into $L^2(\Rn, L^2(dt/t^{2\alpha+1}))$ isometrically (modulo a
multiplicative constant). The second step consists in showing that
its kernel satisfies H\"{o}rmander's condition, after which one appeals
to a well known result of Benedek, Calder\'{o}n and Panzone on vector
valued singular integrals to finish the proof.

\subsection{The fundamental solution of \boldmath$(-\Delta)^{\alpha/2}$}

 Let $I_\alpha$ be the
fundamental solution of $(-\Delta)^{\alpha/2},$ that is, $I_\alpha$
is a function such that $ \hat{I_\alpha}(\xi)= |\xi|^{-\alpha}$ and
is normalized prescribing some behavior at $\infty.$ It is crucial
for our proof to have an explicit expression for $I_\alpha.$ The
result is as follows (see \cite{ACL} or \cite[p. 3699]{MOPV}).

If $\alpha$ is not integer  then
\begin{equation}\label{eq19}
I_\alpha(x) = c_{\alpha,n}\;|x|^{\alpha-n},\quad x \in \Rn\,,
\end{equation}
 for some constant $c_{\alpha,n}$ depending only on $\alpha$ and
$n.$

The same formula works if $\alpha$ is an even integer and the
dimension is odd or if $\alpha$ is an odd integer and the dimension
is even.

The remaining cases, that is, $\alpha$ and $n$ are even integers or
$\alpha$ and $n$ are odd integers are special in some cases. If
$\alpha < n$ formula \eqref{eq19} still holds, but if $\alpha$ is of
the form $n+ 2 N,$ for some non-negative integer $N$, then
$$
I_\alpha(x) = c_{\alpha,n}\,|x|^{\alpha-n} \,(A+B\log|x\,|),\quad x
\in \Rn\,,
$$
where $c_{\alpha,n},$ $A$ and $B$ are constants depending on
$\alpha$ and $n$, and $B \neq 0.$ Thus in this cases (and only in
this cases) a logarithmic factor is present. For instance, if
$\alpha=n$, then $I_\alpha(x) = B \log|x\,|.$ If $n=1$ and
$\alpha=2$, then $I_2(x) = -(1/2)\, |x\,|$ and there is no
logarithmic factor.


\subsection{The case \boldmath$p=2$}\label{3.1}
Given a positive real number $\alpha$ let $N$ be the unique integer
satisfying $2N \le \alpha  < 2N+2.$ Define the square function
associated with $\alpha$ by
\begin{equation}\label{eq20}
S_\alpha(f)^2(x) = \int_0^\infty \left| \fint_{B(x,\,t)}
\frac{\rho_N(y,x)}{t^\alpha}\,dy \right|^2\,\frac{dt}{t}, \quad x
\in \Rn\,,
\end{equation}
where $\rho_N(y,x)$ is
$$
 f(y)-f(x)- \frac{1}{2n} \Delta f(x)\,|y-x|^2 -\dotsb-
\frac{1}{L_{N-1}} \Delta^{N-1}f(x)\,|y-x|^{2(N-1)}- \frac{1}{L_{N}}
(\Delta^{N} f)_{B(x,\,t)}|y-x|^{2N}
$$
if $\alpha = 2N\,,$ and
$$
 f(y)-f(x)- \frac{1}{2n} \Delta f(x)\,|y-x|^2 -\dotsb-
\frac{1}{L_{N-1}} \Delta^{N-1}f(x)\,|y-x|^{2(N-1)}- \frac{1}{L_{N}}
\Delta^{N}f(x)\,|y-x|^{2N}
$$
if $2N < \alpha < 2N+2\,.$

 Recall
that $L_j = \Delta^j (|x\,|^{2j})$ and that the role which the $L_j$
play in Taylor expansions was discussed just before the statement of
Theorem \ref{T3} in the introduction.

 In this
subsection we prove that
\begin{equation}\label{eq21}
\|S_\alpha(f)\|_2 = c \, \|(-\Delta)^{\alpha/2}(f)\|_2\,.
\end{equation}

We first consider the case $\alpha = 2N$ and then we indicate how to
proceed in the (simpler) case $ 2N < \alpha < 2N+2\,.$ Our plan is
to integrate in $x$ in \eqref{eq20}, interchange the integration in
$x$ and $t$ and then apply Plancherel in $x.$ Before we remark that
by making the change of variables $y= x+t h$ we transform integrals
on $B(x,t)$ in integrals on $B(0,1)$ and we get
\begin{equation*}
\begin{split}
 \fint_{B(x,\,t)} \rho_N(y,x) \,dy &=  \fint_{B(0,\,1)}
f(x+th)\,dh - \sum_{j=0}^{N-1} \frac{\Delta^j f
(x)}{L_j}\,t^{2j}\,\fint_{B(0,\,1)}|h|^{2j}\,dh\\*[5pt] &\quad -
\fint_{B(0,\,1)} \Delta^N
f(x+th)\,dh\;t^{2N}\,\fint_{B(0,\,1)}|h|^{2N}\,dh\,.
\end{split}
\end{equation*}
Now apply Plancherel in $x$,\, as explained before, and make the
change of variables $\tau = t \,|\xi|,$ where $\xi$ is the variable
in the frequency side. We obtain
\begin{equation}\label{eq21b}
\|S_\alpha(f)\|_2^2 = c \,I\,\|(-\Delta)^{\alpha/2} f\|_2^2\,,
\end{equation}
where
$$
I=\int_0^\infty \left| F(\tau)- \sum_{j=0}^{N-1} (-1)^{j} \tau^{2j}
\frac{M_j}{L_j}- (-1)^{N} \tau^{2N}
\frac{M_N}{L_N}\,F(\tau)\right|^2 \,
\frac{d\tau}{\tau^{2\alpha+1}}\,.
$$
Here $F$ is a function defined on $[0,\infty)$ such that $F(|\xi|)$
gives the Fourier transform of the radial function
$\frac{1}{|B(0,1)|} \chi_{B(0,1)}$ at the point $\xi$, and we have
introduced the notation $M_j = \fint_{B(0,1} |h|^{2j}\,dh.$ We have
to show that the integral $I$ is finite.

Using the series expansion of the exponential we see that, as $\tau
\rightarrow 0$,
\begin{equation*}
\begin{split}
F(\tau) & = \fint_{B(0,1)} \exp{(i h_1 \tau)}\,dh \\*[5pt] &
=1+\dotsb+(-1)^N\,\tau^{2N} \frac{1}{(2N) !}\,\fint_{B(0,1)}
h_1^{2N}\,dh +O(\tau^{2N+2})\,.
\end{split}
\end{equation*}
We need to compare $\fint_{B(0,1)} h_1^{2N}\,dh$ with
$\fint_{B(0,1)} |h|^{2N}\,dh .$  The linear functionals $P
\rightarrow \Delta^{2j}(P)$ and $P \rightarrow \fint_{B(0,1)} P ,$
defined on the space $H_{2j}$ of homogeneous polynomials of
degree~$2j$, have the same kernel. This follows from the discussion
before the statement of Theorem~\ref{T3} in the introduction.
Therefore, for some constant $c,$
$$
\Delta^{2j}(P) = c \,\fint_{B(0,1)} P , \quad P \in H_{2j}\,.
$$
Taking $P(x)=|x\,|^{2j}$ we get $L_j = c\, \fint_{B(0,1)} |x|^{2j} =
\,dx ,$ and taking $P(x)=x_1^{2j}$ we get \\ $ (2j) ! = c\,
\fint_{B(0,1)} x_1^{2j}\,dx .$ Hence
$$
\frac{1}{(2j)!}\,\fint_{B(0,1)} x_1^{2j}\,dx =
\frac{1}{L_j}\,\fint_{B(0,1)} |x\,|^{2j}\,dx = \frac{M_j}{L_j}\,,
$$
and thus, owing to the definition of $I$ and the fact that
$F(\tau)=1+O(\tau^2),$ as $\tau \rightarrow 0,$
$$
I = \int_0^\infty
O(\tau^{2(2N+2)})\,\frac{d\tau}{\tau^{2\alpha+1}},\quad
{\rm{as}}\;\tau \rightarrow 0\,.
$$
Then $I$ is convergent at $0$ because $\alpha < 2N+2$ (indeed, now
$\alpha = 2N\,,$ but this part of the argument works for the full
range $2N \le \alpha < 2N+2\,.$)

We turn to the case $\tau \rightarrow \infty.$ Notice that the only
difficulty is the last term in the integrand of $I$, because
$$
\int_1^\infty {\tau^{4j}} \frac{d\tau }{\tau^{2\alpha+1}}\,< \infty,
\quad 0 \leq j \leq N-1\,,
$$ provided  $2N \le \alpha.$
To deal with the term
\begin{equation}\label{eq22}
\int_1^\infty |\tau^{2 N}\,F(\tau)|^2
\,\frac{d\tau}{\tau^{2\alpha+1}}
\end{equation}
we only need to recall that $F$ can be expressed in terms of Bessel
functions. Concretely, one has (\cite[Appendix B.5, p. 429]{Gr})
$$
|B(0,1)|\; F(\tau) = \frac{J_{n/2}(\tau)}{|\tau|^{n/2}}\,.
$$
The asymptotic behaviour of $J_{n/2}(\tau)$ gives the inequality, as
$\tau \rightarrow \infty,$
$$
|F(\tau)| \le C\, \frac{1}{\tau^{\frac{n+1}{2}}}\le C\,
\frac{1}{\tau}\,,
$$
which shows that the integral \eqref{eq22} is finite provided
$2N-1 < \alpha,$ which is the case because $\alpha=2N\,.$ Observe
that the argument works for $2N < \alpha < 2N+2$ provided
$\rho_N(y,x)$ is defined replacing  $\Delta^{N}f(x)$ by
$(\Delta^{N} f)_{B(x,\,t)}$.

In the case $2N < \alpha < 2N+2$ we argue similarly. After applying
Plancherel we obtain \eqref{eq21b} where now $I$ is
$$
I=\int_0^\infty \left| F(\tau)- \sum_{j=0}^{N} (-1)^{j} \tau^{2j}
\frac{M_j}{L_j}\right|^2 \, \frac{d\tau}{\tau^{2\alpha+1}}\,.
$$
The proof that $I$ is finite  as $\tau \rightarrow 0$ is exactly as
before. As $\tau \rightarrow \infty$ the situation now is simpler
because the worst term in $I$ is the last one, namely,
$$
\int_1^\infty {\tau^{4N}} \frac{d\tau }{\tau^{2\alpha+1}}\,,
$$
which is finite because $2N < \alpha .$

\subsection{A vector valued operator and its kernel}

Given $f \in W^{\alpha,p},$  there exists a function $g \in L^p$
such that $f = I_\alpha \ast g.$ Indeed, $g =
(-\Delta)^{\alpha/2}(f) .$ Then
\begin{equation*}\label{eq}
\fint_{B(x,\,t)} \rho_N(y,x)\,dy = (K_t \ast g)(x)\,,
\end{equation*}
where the kernel $K_t(x)$ is
\begin{equation}\label{eq23}
K_t(x) = \fint_{B(x,\,t)} \left(I_\alpha(y)-\sum_{j=0}^{N-1}
\frac{1}{L_j}\,\Delta^jI_\alpha(x)\,|y-x|^{2j}- \frac{1}{L_{N}}
(\Delta^{N} I_\alpha)_{B(x,\,t)}\,|y-x|^{2N}\right)\, dy
\end{equation}
if $\alpha = 2N\,,$ and
\begin{equation}\label{eq23b}
K_t(x) = \fint_{B(x,\,t)} \left(I_\alpha(y)-\sum_{j=0}^{N}
\frac{1}{L_j}\,\Delta^jI_\alpha(x)\,|y-x|^{2j}\right)\, dy
\end{equation}
if $ 2N < \alpha < 2N+2 \,.$

Hence the square function associated with the smoothness index
$\alpha$ is
\begin{equation*}\label{eq}
S_\alpha(f)^2(x) = \int_0^\infty \left|(K_t \ast g)(x)\right|^2
\,\frac{dt}{t^{2\alpha+1}}, \quad x \in \Rn \,.
\end{equation*}
Define an operator $T$ acting on functions $f \in L^2(\Rn)$ by
$$
Tg(x) = (K_t \ast g)(x), \quad x \in \Rn\,.
$$
The identity \eqref{eq21} in subsection 4.2 says that $T$ takes
values in $L^2(\Rn, L^2(dt/t^{2\alpha+1}))$ and, more precisely,
that
$$
\int_{\Rn} \|Tg(x) \|_{L^2(dt/t^{2\alpha+1})}^2 \,dx =
\|S_\alpha(f)\|_2^2 = c\,\|g\|_2^2 \,.
$$
Therefore $T$ is an operator mapping isometrically (modulo a
multiplicative constant) $L^2(\Rn)$ into $L^2(\Rn,
L^2(dt/t^{2\alpha+1}))$ and we have an explicit expression for its
kernel. If we can prove that $K_t(x)$ satisfies H\"{o}rmander's
condition
$$
\int_{|x|\geq 2 |y|} \|K_t(x-y)-K_t(x)\|_{L^2(dt/t^{2\alpha+1})}
\,dx\le C,\quad y \in \Rn\,,
$$
then the proof is finished by appealing to a well known result of
Benedek, Calder\'{o}n and Panzone (\cite[Theorem 3.4, p.~492]{GR}; see
also \cite[p.~507]{GR}). In fact, we will show the following
stronger version of H\"{o}rmander's condition
\begin{equation}\label{eq24}
\|K_t(x-y)-K_t(x)\|_{L^2(dt/t^{2\alpha+1})} \le
C\,\frac{|y|^\gamma}{|x|^{n+\gamma}},\quad  |x|\geq 2|y|\,,
\end{equation}
for some $\gamma > 0$ depending on $\alpha$ and $n.$

The proof of \eqref{eq24} is lengthy. In the next subsection we will
consider the case of small ``increments" in $t$, namely $t < |x|/3
.$

\subsection{H\"{o}rmander's condition: \boldmath$t < |x\,|/3 $}

We distinguish two cases: $2N < \alpha < 2N+2$ and $\alpha = 2N$.
Assume first that $2N < \alpha < 2N+2\,.$ Set
$$\chi(x)= \frac{1}{|B(0,1)|}\, \chi_{B(0,1)}(x)$$
and
$$
\chi_t(x)= \frac{1}{t^n} \chi(\frac{x}{t})\,.
$$
To compute the gradient of $K_t$ we remark that
$$
K_t(x)= (I_{\alpha} \ast \chi_t)(x)-\sum_{j=0}^{N} \frac{M_j}{L_j}\,
t^{2j}\,\Delta^jI_{\alpha}(x)\,,
$$
where $M_j = \fint_{B(0,1)} |z|^{2j}\,dz.$ Thus
$$
\nabla K_t(x) = \fint_{B(x,\,t)} \left(\nabla
I_{\alpha}(y)-\sum_{j=0}^{N} \frac{1}{L_j}\,\Delta^j (\nabla
I_{\alpha})(x)\,|y-x|^{2j}\right)\,dy.
$$
Let $P_m(F,x)$ stand for the Taylor polynomial of degree $m$ of the
function $F$ around the point $x.$ Therefore
$$
\nabla K_t (x) = \fint_{B(x,\,t)} \left(\nabla
I_{\alpha}(y)-P_{2N+1}(\nabla I_{\alpha},x)(y)\right)\,dy\,,
$$
because the terms which have been added have zero integral on the
ball $B(x,t)$, either because they are Taylor homogeneous
polynomials of $\nabla I_{\alpha}$ of odd degree or because they are
the ``zero integral part" of a Taylor homogeneous polynomial of
$\nabla I_{\alpha}$ of even degree (see the discussion before the
statement of Theorem $3$ in the introduction). Given $x$ and~$y$
such that $|x|\geq 2 |y|$, apply the formula above to estimate
$\nabla K_t (z)$ for $z$ in the segment from~$x-y$ to~$y.$ The
standard estimate for the Taylor remainder gives
$$
|\nabla K_t (z)| \leq t^{2N+2}\, \sup_{w \in B(z,\,t)} |\nabla
^{2N+3} I_{\alpha}(w)| \,.
$$
Notice that if $z \in [x-y,y]$, $w \in B(z,t)$ and $t \leq |x|/3,$
then $|w| \geq |x|/6 .$ Now, one has to observe that
$$
|\nabla ^{2N+3} I_{\alpha}(w)| \leq C\, |w|^{\alpha-n-2N-3}\,,
$$
owing to the fact that logarithmic factors do not appear because the
exponent~$\alpha-n-2N-3 < -n-1$ is negative. By the mean value
Theorem we then get
$$
|K_t(x-y)-K_t (x)| \leq |y|\,\sup_{z \in [x-y,y]} \,|\nabla K_t (z)|
\leq C \,|y|\,t^{2N+2}\,|x|^{\alpha-n-2N-3}\,.
$$
Since
$$
\left(\int_0^{|x|/3}
t^{2(2N+2)}\,\frac{dt}{t^{2\alpha+1}}\right)^{1/2} = C \,
|x|^{2N+2-\alpha}\,,
$$
we obtain
\begin{equation*}
\left(\int_0^{|x|/3}|K_t (x-y)-K_t(x)|^2
\,\frac{dt}{t^{2\alpha+1}}\right)^{1/2} \leq
C\,\frac{|y|}{|x|^{n+1}}\,,
\end{equation*}
which is  the stronger form of H\"{o}rmander's condition \eqref{eq24}
with $\gamma\!=\!1$ in the domain~$ t\! < \!|x|/3$.\!


Let us consider now the case $\alpha=2N.$ Since $\Delta^N I_{2N}$ is
the Dirac delta at $0,$ $\Delta^N I_{2N}(x)= 0.$ Hence $K_t(x) =
K_t^{(1)}(x) - K_t^{(2)}(x),$ where $K_t^{(1)}$ is given by
\eqref{eq23b} with $\alpha$ replaced by $2N$ and
$$
K_t^{(2)}(x) = \fint_{B(x,\,t)}  \frac{1}{L_{N}} (\,(\Delta^{N}
I_{2N})_{B(x,\,t)} \,|y-x|^{2N}\,dy =
\frac{M_N}{L_N}\,t^{2N}\,(\Delta^{N} I_{2N})_{B(x,\,t)}\,.
$$
The kernel $K_t^{(1)}$ is estimated exactly as in the first case by
just setting $\alpha = 2N.$ The kernel $K_t^{(2)}$ requires a
different argument.

Since $\Delta^N I_{2N} $ is the Dirac delta at the origin,
 $K_t^{(2)}$ is a constant multiple of $t^{2N}\, \chi_t.$  We show now that this kernel
 satisfies the strong form of H\"{o}rmander's condition. The quantity
$|\chi_t(x-y)-\chi_t(x)|$ is non-zero only if $|x-y|< t< |x|$ or
$|x|< t< |x-y|,$ in which cases takes the value $1/c_n \,t^n$, $c_n
= |B(0,1)|.$ On the other hand, if $|x|\geq 2 |y|\,$ then each $z$
in the segment joining $x$ and $x-y$ satisfies $|z|\geq |x|/2.$
Assume that $|x-y|<  |x|$ (the case~$|x|< |x-y|$ is similar). Then
\begin{equation*}
\begin{split}
\left(\int_{0}^\infty (t^{2N} \,(\chi_t(x-y)-\chi_t(x)))^2
\,\frac{dt}{t^{4N+1}} \right)^{\frac{1}{2}} & =
C\,\left(\int_{|x-y|}^{|x|}
\frac{dt}{t^{2n+1}}\right)^{\frac{1}{2}}\\*[5pt] & = C\,\left(
\frac{1}{|x-y|^{2n}}-\frac{1}{|x|^{2n}}\right)^{\frac{1}{2}} \le C\,
\frac{|y|^{1/2}}{|x|^{n+1/2}}\,,
\end{split}
\end{equation*}
which is \eqref{eq24} with $\gamma = 1/2.$

\subsection{H\"{o}rmander's condition: \boldmath$t \geq |x\,|/3 $}
We distinguish three cases: $\alpha < n+1$, $\alpha = n+1$ and
$\alpha > n+1$.

If $\alpha < n+1$, then all terms in the expressions \eqref{eq23}
and  \eqref{eq23b} defining $K_t$ satisfy H\"{o}rmander's condition in
the domain $t\! \geq \!|x\,|/3 .$  If $\alpha =2N$ the last term in
\eqref{eq23} is of the form
$$
- \fint_{B(x,\,t)} \frac{1}{L_{N}} (\Delta^{N}
I_\alpha)_{B(x,\,t)}\,|y-x|^{2N}\, dy =
-\frac{M_N}{L_N}\,t^{2N}\,(\Delta^{N} I_{\alpha})_{B(x,\,t)} = C\,
t^{2N}\,\chi_t(x)\,,
$$
which has been dealt with in the previous subsection. Let us
consider the terms of the form ~$t^{2j}\,\Delta^j I_\alpha(x)$, $j
\geq 0.$ One has the gradient estimate
\begin{equation}\label{eq27}
|t^{2j}\,\nabla \Delta^j I_\alpha(x)| \le C\, t^{2j}\,|x|^{\alpha-n
-2j-1}\,,
\end{equation}
because no logarithmic factors appear, the reason being that the
exponent $\alpha-n -2j-1 \leq \alpha -(n+1)$ is negative. Since
$$
\left(\int_{|x|/3}^\infty
t^{2(2j)}\frac{dt}{t^{2\alpha+1}}\right)^{1/2}=
C\,|x|^{2j-\alpha}\,,
$$
we get H\"{o}rmander's condition with $\gamma=1$ in the domain $t \geq
|x\,|/3 .$

It remains to take care of the first term $\fint_{B(x,\,t)}
I_\alpha(y)\,dy$ in \eqref{eq23}. We have the following obvious
estimate for its gradient
$$ \left|\fint_{B(x,\,t)} \nabla
I_\alpha(y)\,dy \right| \le C\, \fint_{B(x,\,t)}
|y|^{\alpha-n-1}\,dy\,.
$$
Notice that there are no logarithmic factors precisely because
$\alpha < n+1.$ The integrand in the last integral is locally
integrable if and only if $\alpha > 1.$ Assume for the moment that
$1 < \alpha <n+1.$ Then
$$
\left| \fint_{B(x,\,t)} \nabla I_\alpha(y)\,dy \right| \le C\,
t^{-n+\alpha-1}\,.
$$
Since
$$
\left(\int_{|x|/3}^\infty
t^{2(\alpha-n-1)}\frac{dt}{t^{2\alpha+1}}\right)^{1/2}=
C\,|x|^{-n-1}\,,
$$
we get H\"{o}rmander's condition with $\gamma=1$ in the domain $t \geq
|x\,|/3 .$ The case $\alpha=1$ has been treated in section 1, so we
can assume that $0 < \alpha < 1 .$ We need the following well known
and easily proved inequality
\begin{CZ}
Let $E$ be a measurable subset of $\Rn$ and $0 < \beta < n .$ Then
$$
\int_E \frac{dz}{|z|^{n-\beta}} \le C\,|E|^{\beta/n}\,,
$$
where $|E|$ is the Lebesgue measure of $E$.
\end{CZ}
Denoting by $D$ the symmetric difference between $B(x,t)$ and
$B(x-y,t)$, we obtain, by the Lemma,
\begin{equation*}
\begin{split}
|\fint_{B(x-y,\,t)} I_\alpha(y)\,dy - \fint_{B(x,\,t)}
I_\alpha(y)\,dy| & \le C\,t^{-n}\,\int_D
\frac{dy}{|y|^{n-\alpha}}\\*[5pt] & \le C\,
t^{-n}\,(t^{n-1}\,|y|)^{\alpha/n} =
C\,t^{\alpha-n-\alpha/n}\,|y|^{\alpha/n}\,.
\end{split}
\end{equation*}
Since
$$
\left(\int_{|x|/3}^\infty
t^{2(\alpha-n-\alpha/n)}\frac{dt}{t^{2\alpha+1}}\right)^{1/2}=
C\,|x|^{-n-\alpha/n}\,,
$$
we get H\"{o}rmander's condition with $\gamma=\alpha/n$ in the domain $t
\geq |x\,|/3 .$

We tackle now the case $\alpha=n+1 .$ Since $\alpha$ and $n$ are
integers with different parity no logarithmic factor will appear in
$I_\alpha.$ Thus $I_\alpha(x)= C\, |x\,|.$  The proof above shows
that the terms $t^{2j}\,\Delta^j I_\alpha(x)$ appearing in the
expression \eqref{eq23} of the kernel $K_t$ still satisfy
H\"{o}rmander's condition for $j \geq 1$. The remaining term is
$$
\fint_{B(x,\,t)} \left(I_\alpha(y)-I_\alpha(x)\right)\,dy
$$
and its gradient is estimated by remarking that the function $|x\,|$
satisfies a Lipschitz condition. We obtain
$$
\left| \fint_{B(x,\,t)} \left(\nabla I_\alpha(y)- \nabla
I_\alpha(x)\right)\,dy\, \right| \le C\,.
$$
But clearly
$$
\left(\int_{|x|/3}^\infty \frac{dt}{t^{2\alpha+1}}\right)^{1/2}=
C\,|x|^{-\alpha} = C\,|x|^{-n-1}\,,
$$
which completes the argument.

We turn our attention to the case $\alpha > n+1.$  If $\alpha=2N$
the part of $K_t$ which has to be estimated is
$$
H_t(x) = \fint_{B(x,\,t)} \left(I_\alpha(y)-\sum_{j=0}^{N-1}
\frac{1}{L_j}\,\Delta^jI_\alpha(x)\,|y-x|^{2j}\right)\,dy\,.
$$
Since $\Delta^{2N}I_\alpha$ is the Dirac delta at the origin and $x
\neq 0$ we have $\Delta^{2N}I_\alpha (x) =0$ and so $H_t(x)$
coincides with the expression \eqref{eq23b} for $K_t(x)$ in the case
$2N < \alpha  < 2N+2 \,.$  We are then going to deal with this
kernel in the full range $2N \leq \alpha < 2N+2$. Let $M$ be the
unique positive integer $M$ such that $ -1< \alpha-n-2M \le 1 .$ We
split $K_t$ into two terms according to $M$, that is, $K_t =
K_t^{(1)}-K_t^{(2)},$ where
$$
K_t^{(1)}(x)= \fint_{B(x,\,t)} \left(I_\alpha(y)-\sum_{j=0}^{M-1}
\frac{1}{L_j}\,\Delta^jI_\alpha(x)\,|y-x|^{2j}\right)\,dy
$$
and
$$
K_t^{(2)}(x)= \fint_{B(x,\,t)} \left(\sum_{j=M}^{N}
\frac{1}{L_j}\,\Delta^jI_\alpha(x)\,|y-x|^{2j}\right)\,dy\,.
$$
The estimate of each of the terms in $K_t^{(2)}$ is performed as we
did for the case $\alpha < n+1$ (if $\alpha  = 2N$, then $j \le N-1
$).  The gradient estimate is exactly \eqref{eq27}. Now no
logarithmic factors appear because the exponent satisfies $\alpha-n
-2j-1 \leq \alpha -n -2M-1 \le 0$. The rest is as before.

To estimate $K_t^{(1)}$ we distinguish three cases: $-1< \alpha -n
-2M < 0$, \,$0 < \alpha -n -2M \le 1$ and $\alpha -n -2M =0.$ In
the first case we write the gradient of $K_t^{(1)}$ as
\begin{equation*}
\begin{split}
\nabla K_t^{(1)}(x) & = \fint_{B(x,\,t)} \left(\nabla
I_{\alpha}(y)-\sum_{j=0}^{M-1} \frac{1}{L_j}\,\Delta^j (\nabla
I_{\alpha})(x)\,|y-x|^{2j}\right)\,dy \\*[5pt] & = \fint_{B(x,\,t)}
\left(\nabla I_{\alpha}(y)- P_{2M-2}(\nabla
I_\alpha,x)(y)\,\right)\,dy \,,
\end{split}
\end{equation*}
where $P_{2M-2}$ is the Taylor polynomial of degree $2M-2$ of
$\nabla I_\alpha$ around the point $x.$ As before, the added terms
have zero integral on $B(x,t)$ either because they are homogeneous
Taylor polynomials of odd degree or the ``zero integral part" of
homogeneous Taylor polynomials of even degree. Now fix $y$ in
$B(x,t)$ but not in the half line issuing from $x$ and passing
through the origin. Define a function $g$ on the interval $[0,1]$ by
$$
g(\tau)= \nabla I_{\alpha}(x+\tau(y-x))- P_{2M-2}(\nabla
I_\alpha,x)(x+\tau(y-x)),\quad 0 \leq \tau \leq 1\,.
$$
Then $g \in C^\infty[0,1]$ because the segment with endpoints $x$
and $y$ omits the origin. Since $g^{j)}(0)=0,$ $0 \leq j \leq 2M-2
,$
\begin{equation*}
\begin{split}
\nabla I_{\alpha}(y)- P_{2M-2}(\nabla I_\alpha,x)(y) & = g(1)-
\sum_{j=0}^{2M-2}\frac{g^{j)}(0)}{j!}\\*[5pt] & = \int_0^1
\frac{(1-\tau)^{2M-2}}{(2M-2)!} \,g^{2M-1)}(\tau)\,d\tau \,,
\end{split}
\end{equation*}
by the integral form of Taylor's remainder. The obvious estimate for
the derivative of $g$ of order $2M-1$ is
$$
|g^{2M-1)}(\tau)| \le |\nabla^{2M-1} \nabla I_\alpha
(x+\tau(y-x))||y-x|^{2M-1} \le C\,
\frac{t^{2M-1}}{|x+\tau(y-x)|^{n-(\alpha-2M)}} \,.
$$
Since we are in the first case, $\alpha$ is not integer and thus no
logarithmic factor exists. Moreover $0 < n-(\alpha-2M)< 1,$ which
implies that and that $1/|z\,|^{ n-(\alpha-2M)}$ is locally
integrable in any dimension. Therefore
\begin{equation*}
\begin{split}
|\nabla K_t^{(1)}(x)| & \le C\, \int_0^1 \left(
t^{2M-1-n}\,\int_{B(x,t)}
\frac{dy}{|x+\tau(y-x)|^{n-(\alpha-2M)}}\right)\,d\tau\,,\\*[5pt] &
=C\, t^{2M-1-n} \int_0^1 \left(\int_{B(x,t\,\tau\,)}
\frac{dz}{|z\,|^{n-(\alpha-2M)}}\right) \frac{d\tau}{\tau^n}\\*[5pt]
&\le C\,t^{2M-1-n} \int_0^1
(t\,\tau)^{\alpha-2M}\,\frac{d\tau}{\tau^n}\\*[5pt] & =
t^{\alpha-n-1}\,\int_0^1 \frac{d\tau}{\tau^{n-(\alpha-2M)}} =
C\,t^{\alpha-n-1} \,.
\end{split}
\end{equation*}
Since
$$
\left(\int_{|x|/3}^\infty
t^{2(\alpha-n-1)}\frac{dt}{t^{2\alpha+1}}\right)^{1/2}=
C\,|x|^{-n-1}\,,
$$
we get H\"{o}rmander's condition with $\gamma=1$ in the domain $t \geq
|x\,|/3 .$

Let us consider the second case: $0 < \alpha -n -2M \le 1.$ This
time we express the gradient of $K_t^{(1)}$ by means of a Taylor
polynomial of degree $2M-1$:
\begin{equation*}
\begin{split}
\nabla K_t^{(1)}(x) & = \fint_{B(x,\,t)} \left(\nabla
I_{\alpha}(y)-\sum_{j=0}^{M-1} \frac{1}{L_j}\,\Delta^j (\nabla
I_{\alpha})(x)\,|y-x|^{2j}\right)\,dy \\*[5pt] & = \fint_{B(x,\,t)}
\left(\nabla I_{\alpha}(y)- P_{2M-1}(\nabla
I_\alpha,x)(y)\,\right)\,dy \,.
\end{split}
\end{equation*}
Using again the integral form of the Taylor remainder of the
function $g$, with $P_{2M-2}$ replaced by $P_{2M-1},$ we obtain
\begin{equation*}
\begin{split}
|\nabla K_t^{(1)}(x)| & \le C\, \int_0^1 \left(
t^{2M-n}\,\int_{B(x,t)}
\frac{dy}{|x+\tau(y-x)|^{n-(\alpha-2M-1)}}\right)\,d\tau\,,\\*[5pt]
& =C\, t^{2M-n} \int_0^1 \left(\int_{B(x,t\,\tau\,)}
\frac{dz}{|z\,|^{n-(\alpha-2M-1)}}\right)
\frac{d\tau}{\tau^n}\\*[5pt] &\le C\,t^{2M-n} \int_0^1
(t\,\tau)^{\alpha-2M-1}\,\frac{d\tau}{\tau^n}\\*[5pt] & =
t^{\alpha-n-1}\,\int_0^1 \frac{d\tau}{\tau^{n-(\alpha-2M-1)}} =
C\,t^{\alpha-n-1} \,,
\end{split}
\end{equation*}
from which we get the desired estimate as before.

We turn now to the last case left, $\alpha=n+2M,$ with $M$ a
positive integer. In this case
$$
I_\alpha(x) = C\,|x|^{2M} \,(A+B\log|x\,|),\quad x \in \Rn,\quad B
\neq 0\,,
$$
where $A$, $B$ and $C$ are constants depending on $n$ and $M.$ We
also have
$$
\Delta^{M-1}I_\alpha(x) = C\,|x|^{2} \,(A_1+B_1 \log|x\,|),\quad x
\in \Rn
$$
and
$$
\nabla \Delta^{M-1}I_\alpha(x) = C\,x \;(A_2+B_2 \log|x\,|),\quad x
\in \Rn \,.
$$
In particular $\nabla \Delta^{M-1}I_\alpha$ is in the Zygmund class
on $\R^n.$ We have
\begin{equation*}
\begin{split}
\nabla K_t^{(1)}(x) & = \fint_{B(x,\,t)} \!\left(\nabla
I_{\alpha}(y)-\!\sum_{j=0}^{M-2} \frac{1}{L_j}\,\Delta^j (\nabla
I_{\alpha})(x)\,|y\!-\!x|^{2j} - \frac{1}{L_{M-1}}\,\Delta^{M-1}
(\nabla I_{\alpha})(x)\,|y\!-\!x|^{2M-2}\right)\,dy \\*[5pt] & =
\fint_{B(x,\,t)}\! \left(\nabla I_{\alpha}(y)- P_{2M-3}(\nabla
I_\alpha,x)(y) - \frac{1}{L_{M-1}}\,\Delta^{M-1} (\nabla
I_{\alpha})(x)\,|y-x|^{2M-2}\,\right) \,dy \,.
\end{split}
\end{equation*}
Introduce the function $g$ as above, with $P_{2M-2}$ replaced by
$P_{2M-3},$ so that
\begin{equation*}
\begin{split}
\nabla I_{\alpha}(y)- P_{2M-3}(\nabla I_\alpha,x)(y) & = g(1)-
\sum_{j=0}^{2M-3}\frac{g^{j)}(0)}{j!}\\*[5pt] & = \int_0^1
 (2M-2) (1-\tau)^{2M-3} \, \frac{g^{2M-2)}(\tau)}{(2M-2)!}\,d\tau
 \,.
\end{split}
\end{equation*}
Now
\begin{equation*}
\begin{split}
\frac{g^{2M-2)}(\tau)}{(2M-2)!} & = \sum_{|\beta|=2M-2}
\left(\frac{\partial^\beta \nabla I_\alpha (x+ \tau(y-x))}{\beta
!}\right) \,(y-x)^\beta \\*[5pt] & = \sum_{|\beta|=2M-2}
\left(\partial^\beta \nabla I_\alpha (x+ \tau(y-x)) -
\partial^\beta \nabla I_\alpha (x)
 \right)\, \frac{(y-x)^\beta}{\beta !}\\*[5pt]
 &\quad+ \sum_{|\beta|=2M-2} \left(\frac{\partial^\beta \nabla I_\alpha
(x)}{\beta !}\right) \,(y-x)^\beta\,.
\end{split}
\end{equation*}
The last term in the preceding equation is the homogeneous Taylor
polynomial of degree~$2M-2$ of the vector $\nabla I_\alpha$ around
the point $x.$ It is then equal to a homogeneous polynomial of the
same degree with zero integral on $B(x,t)$ plus
$\frac{1}{L_{M-1}}\,\Delta^{M-1} (\nabla
I_{\alpha})(x)\,|y-x|^{2M-2}$ (by the discussion before the
statement of Theorem 3 in the introduction). Hence
$$
\int_{B(x,t\,)} \left(\sum_{|\beta|=2M-2} \left(\frac{\partial^\beta
\nabla I_\alpha (x)}{\beta !}\right) \,(y-x)^\beta -
\frac{1}{L_{M-1}}\,\Delta^{M-1} (\nabla
I_{\alpha})(x)\,|y-x|^{2M-2}\right)\,dy =0 \,,
$$
and therefore, remarking that $\int_0^1 (2M-2) (1-\tau)^{2M-3}
\,d\tau =1,$
\begin{multline*}
\nabla K_t^{(1)}(x) \\
=\! \fint_{B(x,\,t)}\int_0^1 \!(2M-2) (1-\tau)^{2M-3} \left(
\sum_{|\beta|=2M-2} \left(\partial^\beta \nabla I_\alpha (x\!+\!
\tau(y\!-\!x)) -
\partial^\beta \nabla I_\alpha (x)
 \right) \frac{(y-x)^\beta}{\beta !}  \right)\,d\tau \,dy\,.
\end{multline*}
Thus
\begin{equation*}
|\nabla K_t^{(1)}(x)|  \le  C\,\int_0^1 \sum_{|\beta|=2M-2} \left|
\fint_{B(x,\,t)}  \left(\partial^\beta \nabla I_\alpha (x+
\tau(y-x)) -
\partial^\beta \nabla I_\alpha (x)
 \right)\, \frac{(y-x)^\beta}{\beta !} \,dy   \right|\,d\tau\,.
\end{equation*}
Making the change of variables $h= \tau(y-x)$ the integral in $y$
above becomes
$$
J= \tau^{-|\beta|}\,\fint_{B(0,\,t\,\tau)} \left(\partial^\beta
\nabla I_\alpha (x+ h) -
\partial^\beta \nabla I_\alpha (x)
 \right) \frac{h^\beta}{\beta !}\,dh\,,
$$
which is invariant under the change of variables $h'= - h$, because
$|\beta|$ is even. Hence
$$
 2J= \tau^{-|\beta|}\,\fint_{B(0,\,t\,\tau)} \left(\partial^\beta
\nabla I_\alpha (x+ h)+\partial^\beta \nabla I_\alpha (x- h) - 2\,
\partial^\beta \nabla I_\alpha (x)
 \right) \frac{h^\beta}{\beta !}\,dh\,.
$$
Now we claim that $\partial^\beta \nabla I_\alpha$ is in the Zygmund
class for $|\beta|=2M-2$. This follows from the fact that the
Zygmund class in invariant under homogeneous smooth Calder\'{o}n
-Zygmund operators, $\Delta^{M-1}$ is an elliptic operator and
$\Delta^{M-1}\nabla I_\alpha$ is in the Zygmund class. Hence
$$
|J| \le C\, \tau^{-|\beta|}\,\fint_{B(0,\,t\,\tau)} |h|^{1+|\beta|}
\,dh  \le C\, t^{2M-1} \,\tau \,.
$$
Thus
$$
|\nabla K_t^{(1)}(x)| \le  C\,t^{2M-1}\,.
$$
Since
$$
\left(\int_{|x|/3}^\infty
t^{2(2M-1)}\frac{dt}{t^{2\alpha+1}}\right)^{1/2}= C\,|x|^{-n-1}\,,
$$
we get H\"{o}rmander's condition with $\gamma=1$ in the domain $t \geq
|x\,|/3 .$

\subsection{The sufficient condition}

In this section we prove that condition (2) in Theorem \ref{T3} is
sufficient for $f \in W^{\alpha,p}.$ Let $f, g_1,\dotsc,g_N \in L^p$
satisfy $S_\alpha(f,g_1,\dotsc,g_N) \in L^p. $ Take an infinitely
differentiable function~$\phi \geq 0$ with compact support in
$B(0,1),$ $\int \phi = 1$ and set $\phi_\epsilon(x)=
\frac{1}{\epsilon^n} \phi(\frac{x}{\epsilon}),\; \epsilon>0.$
Consider the regularized functions $f_\epsilon = f \ast
\phi_\epsilon$, $g_{j\,,\epsilon} = g_j \ast \phi_\epsilon$, $1 \le
j \le N. $ We want to show first that the infinitely differentiable
function $f_\epsilon$ is in $W^{\alpha,p}.$ We have
$(-\Delta)^{\alpha/2} f_\epsilon = f \ast (-\Delta)^{\alpha/2}
\phi_\epsilon.$ We need a lemma.

\begin{lemma}\label{laplacia}\quad
\begin{itemize}
\item[(i)] If $\varphi$ is a function in the Schwartz class and $\alpha$
any positive number, then $(-\Delta)^{\alpha/2} \varphi$ belongs to
all $L^q$ spaces , $ 1 \leq q \leq\ \infty $.
\item[(ii)] If $f \in L^p, 1 \le  p \le \infty\,,$ then $(-\Delta)^{\alpha/2}f$ is a tempered
distribution.
\end{itemize}
\end{lemma}

\begin{proof}
Set $\psi= (-\Delta)^{\alpha/2} \varphi$. If $\alpha =2m$ with $m$
a positive integer, then $\psi= (-\Delta)^{m} \varphi$ is in the
Schwartz  class and so the conclusion in $(i)$ follows. If $\alpha
=2m+1$, then
$$
\psi=(-\Delta)^{1/2}(-\Delta)^m \varphi = - i \sum_{j=1}^n
R_j(\,\partial_j\, (-\Delta)^m \varphi\,)\,,
$$
where $R_j$ are the Riesz transforms, that is, the Calder\'{o}n-Zygmund
operators whose\linebreak Fourier multiplier is $ \xi_j /|\xi|.$  It
is clear from the formula above that $\psi$ is infinitely
differentiable on~$\Rn$ and so the integrability issue is only at
$\infty.$ Since $\partial_j\, (-\Delta)^m \varphi\,),$ has zero
integral, one has, as $x \rightarrow \infty$, $|\psi(x)| \le C\,
|x|^{-n-1},$ and so the conclusion follows.

Assume now that $m-1 <  \alpha < m,$  for some positive integer $m.$
Thus
$$\hat{\psi}(\xi) = |\xi|^\alpha \hat{\varphi}(\xi) = |\xi|^m
\hat{\varphi}(\xi) \frac{1}{|\xi|^{m-\alpha}}\,.$$  If $m$ is even,
of the form $m = 2 M$ for some positive integer $M$, then
$$
\psi = \Delta^M \varphi \ast I_{m-\alpha}\,,
$$
where $I_{m-\alpha}(x)= C\, |x|^{m-\alpha -n}.$ Hence $\psi$ is
infinitely differentiable on $\Rn.$ Since $\Delta^M \varphi$ has
zero integral, $|\psi(x)| \le C\, |x|^{m-\alpha-n-1},$ as $x
\rightarrow \infty.$ But $\alpha-m +1
> 0$ and thus $\psi$ is in all $L^q$ spaces.

If $m$ is odd, of the form $m=2M+1$ for some non-negative integer
$M,$ then
$$
\psi = -i \sum_{j=1}^{n} R_j(\partial_j \Delta^M \varphi) \ast
I_{m-\alpha}\,.
$$
Again  $\psi$ is infinitely differentiable on $\Rn$ and, since
$R_j(\partial_j \Delta^M \varphi)$ has zero integral (just look at
the Fourier transform and remark that it vanishes at the origin), we
get $|\psi(x)| \le C\, |x|^{m-\alpha-n-1},$ as $x \rightarrow
\infty,$ which completes the proof of $(i)$.

To prove $(ii)$ take a function $\varphi$ in the Schwartz class. Let
$q$ be the exponent conjugate to $p$. Define the action of
$(-\Delta)^{\alpha/2}f$ on the Schwartz function $\varphi$ as
$\langle f, (-\Delta)^{\alpha/2}\varphi \rangle\,.$ By part $(i)$
and H\"{o}lder's inequality one has
$$
|\langle (-\Delta)^{\alpha/2}f, \varphi \rangle|= |\langle f,
(-\Delta)^{\alpha/2}\varphi \rangle| \le C\, \|f\|_p\,
\|(-\Delta)^{\alpha/2}\varphi\|_q\,,
$$
which completes the proof of $(ii).$
\end{proof}

Let us continue the proof of the sufficiency of condition (2). By
the lemma  $(-\Delta)^{\alpha/2} \phi_\epsilon$ is in $L^1$ and so
$$
\|(-\Delta)^{\alpha/2} f_\epsilon\|_p = \| f \ast
(-\Delta)^{\alpha/2} \phi_\epsilon\|_p \le \|f\|_p \,
\|(-\Delta)^{\alpha/2} \phi_\epsilon\|_1 \,.
$$
Hence $f_\epsilon \in W^{\alpha,p}.$

Next, we claim that
\begin{equation}\label{eq28}
S_\alpha(f_\epsilon, g_{1\,,\epsilon},\dotsc,g_{N\,,\epsilon})(x)
\le (S_\alpha(f,g_1,\dotsc,g_N)\ast \phi_\epsilon)(x),\quad x \in
\Rn\,.
\end{equation}
One has
\begin{equation*}
S_\alpha(f,g_1,\dotsc,g_N)(x) = \|
R_\alpha(x,t)\|_{L^2(dt/t^{2\alpha+1})}\,,
\end{equation*}
where
$$
R_\alpha(x,t) = (f \ast \chi_t)(x)-f(x) - \sum_{j=1}^{N-1}
M_j\,g_j(x)\,t^{2j}- M_N\,(g_N \ast \chi_t)(x)\,t^{2N}
$$
if $\alpha= 2N$ and
$$
R_\alpha(x,t) = (f \ast \chi_t)(x)-f(x) - \sum_{j=1}^{N}
M_j\,g_j(x)\,t^{2j}
$$
if $2N < \alpha < 2N+2\,.$ As before we have set $M_j =
\fint_{B(0,1)} |z|^{2j}\,dz\,.$  Minkowsky's integral inequality now
readily yields \eqref{eq28}.

Set
$$
D_\epsilon(x) = \| \sum_{j=1}^{N-1} M_j \,(\frac{\Delta^j f_\epsilon
(x)}{L_j}- g_{j\,,\epsilon}(x))\,t^{2j}-M_N\, \left((\frac{\Delta^N
f_\epsilon}{L_N}- g_{N\,,\epsilon})\ast \chi_t
\right)(x)\,t^{2N}\|_{L^2(dt/t^{2\alpha+1})}
$$
if $\alpha= 2N$ and
$$
D_\epsilon(x) = \| \sum_{j=1}^{N} M_j \,(\frac{\Delta^j f_\epsilon
(x)}{L_j}- g_{j\,,\epsilon}(x))\,t^{2j})\|_{L^2(dt/t^{2\alpha+1})}
$$
if $2N < \alpha < 2N+2\,.$

By \eqref{eq28}
\begin{equation*}
\begin{split}
D_\epsilon(x) & \le S_\alpha(f_\epsilon)(x) + S_\alpha(f_\epsilon,
g_{1\,,\epsilon},\dotsc,g_{N\,,\epsilon})(x)\\& \le
S_\alpha(f_\epsilon)(x) + (S_\alpha(f,g_1,\dotsc,g_N)\ast
\phi_\epsilon)(x)\,,
\end{split}
\end{equation*}
and so $D_\epsilon \in L^p.$ In particular, $D_\epsilon(x)$ is
finite for almost all $x \in \Rn.$ Thus
$$ \liminf_{t\rightarrow
0} \left|\sum_{j=1}^{N-1} M_j \,(\frac{\Delta^j f_\epsilon
(x)}{L_j}- g_{j\,,\epsilon}(x))\,t^{2j}-M_N\, \left((\frac{\Delta^N
f_\epsilon}{L_N}- g_{N\,,\epsilon})\ast \chi_t
\right)(x)\,t^{2N}\right| \,t^{-\alpha} =0 \,,
$$
for almost all $x \in \Rn$\,, if $\alpha= 2N \,,$ and
$$ \liminf_{t\rightarrow
0} \left|\sum_{j=1}^{N} M_j \,(\frac{\Delta^j f_\epsilon (x)}{L_j}-
g_{j\,,\epsilon}(x))\,t^{2j}\right| \,t^{-\alpha} =0\,,
$$
for almost all $x \in \Rn$\,, if $2N < \alpha < 2N+2\,.$ It is easy
to conclude that the only way this may happen is whenever
$$
\frac{\Delta^j f_\epsilon(x)}{L_j}= g_{j\,,\epsilon}(x), \quad 1 \le
j \le N \,,
$$
for almost all $x \in \Rn.$ Hence
$$
\frac{\Delta^j f_\epsilon}{L_j} \rightarrow g_j, \quad  1 \le j \le
N\,,
$$
in $L^p$ as $ \epsilon \rightarrow 0 .$ Since $f_\epsilon
\rightarrow f$  in $L^p$  as $\epsilon \rightarrow 0 ,$
$$
\Delta^j f_\epsilon \rightarrow \Delta^j f, \quad  1 \le j \le N\,,
$$
in the weak topology of tempered distributions. Hence
$$
\frac{\Delta^j f}{L_j} =  g_j, \quad  1 \le j \le N\,.
$$

We claim now that the functions $f_\epsilon$ are uniformly bounded
in $W^{\alpha,p}\,.$ Indeed, by the proof of necessity of condition
(2) and by \eqref{eq28},
\begin{equation*}
\begin{split}
\| (-\Delta)^{\alpha/2} f_\epsilon\|_p & \simeq
\|S_{\alpha}(f_\epsilon, \Delta f_\epsilon / L_1,\dotsc, \Delta^N
f_\epsilon / L_N)\|_p \\ & \le  \|S_{\alpha}(f, \Delta f /
L_1,\dotsc, \Delta^N f / L_N)\|_p < \infty.
\end{split}
\end{equation*}
Hence there exist a function $h \in L^p$ and a sequence $\epsilon_j
\rightarrow 0$ as $j \rightarrow \infty$ such that
$$
(-\Delta)^{\alpha/2} f_{\epsilon_j} \rightarrow h \quad \text{as}
\quad j \rightarrow \infty
$$
in the weak $\star$ topology of $L^p$. On the other hand, by Lemma
2, $(-\Delta)^{\alpha/2} f$ is a tempered distribution and so
$$
(-\Delta)^{\alpha/2} f_{\epsilon} \rightarrow (-\Delta)^{\alpha/2} f
\quad \text{as} \quad \epsilon \rightarrow 0
$$
in the weak topology of tempered distributions. Therefore
$(-\Delta)^{\alpha/2} f = h \in L^p$ and the proof is complete.

\section{Final remarks}
Let $(X,d,\mu)$ be a metric measure space, that is, $X$ is a metric
space with distance $d$ and~$\mu$ is a Borel measure on $X.$ We
assume that the support of $\mu$ is $X.$ Then, given $\alpha > 0$
and $1 < p < \infty ,$ we can define the Sobolev space
$W^{\alpha,p}(X)$ as follows. Let $N$ be the unique integer such
that $2N \leq \alpha < 2N+2.$  Given locally integrable functions
$f, g_1, \dotsc, g_N$ define a square function by
$$
S_\alpha(f,g_1,g_2,\dotsc,g_N)(x)^2= \int_0^D \left|
\fint_{B(x,\,t)} \frac{R_N(y,x)}{t^\alpha}\,d\mu(y)
\right|^2\,\frac{dt}{t}, \quad x \in \Rn\,,
$$
where $D$ is the diameter of $X$ and $R_N(y,x)$ is
$$
R_N(y,x)= f(y)-f(x)- g_1(x)\,d(y,x)^2 +\dotsb -g_{N-1}(x)\,
d(y,x)^{2(N-1)}- (g_N)_{B(x,\,t)}d(y,x)^{2N}
$$
if $\alpha = 2N$  and
$$
R_N(y,x)= f(y)-f(x)- g_1(x)\,d(y,x)^2 +\dotsb -g_{N-1}(x)\,
d(y,x)^{2(N-1)}- g_N(x)\,d(y,x)^{2N}
$$
if $2N < \alpha < 2N+2\,.$  Here the barred integral stands for the
mean with respect to $\mu$ on the indicated set, $B(x,t)$ is the
open ball with center $x$ and radius $t$ and $g_{B(x,t)}$ is the
mean of the function $g$ on~$B(x,t).$

We say that a function $f$ belongs to the Sobolev space
$W^{\alpha,p}(X)$ provided $f \in L^p(\mu)$ and there exist
functions $g_1, g_2, \dotsc, g_N \in L^p(\mu)$ such that
$S_\alpha(f,g_1,g_2,\dotsc,g_N) \in L^p(\mu).$

We have seen in the previous sections that this definition yields
the usual Sobolev spaces if $X=\Rn$ is endowed with the Euclidean
distance and $\mu$ is Lebesgue measure. One can prove with some
effort that the same is true if $\Rn$ is replaced by a half-space.
Very likely this should also work for smoothly bounded domains, but
we have not gone that far.

There are many interesting questions one may ask about these new
Sobolev spaces. For instance, how do they compare, for $\alpha=1$,
with the known first order Sobolev spaces, notably those introduced
by Hajlasz in \cite{H} or the Newtonian spaces of \cite{S}\, ?  For
higher orders of smoothness one would like to compare them with
those introduced by Liu, Lu and Wheeden in \cite{LLW}. One may also
wonder about their intrinsic properties, namely, about versions of
the Sobolev imbedding theorem, the Poincar\'{e} inequality and so on.

For the Sobolev imbedding theorem the following remark might be
useful. In $\Rn$ the $L^p$ space can be characterized by means of
the following ``zero smoothness" square function:
$$
S_0(f)^2(x)= \int_0^\infty \left| f_{B(x,\,t)} - f_{B(x,\,2t)}
\right|^2\,\frac{dt}{t}, \quad x \in \Rn\,.
$$
The result is then that a locally integrable function $f$ is in
$L^p$ if and only if $S_0(f) \in L^p.$ The proof follows the pattern
described several times in this paper. One first deals with the case
$p=2$ via a Fourier transform computation. Then one introduces a
$L^2(dt/t)$-valued operator $T$ such that
$$
\|T(f)\|_{L^2 (\Rn,\, L^2(dt/t))} = c\,\|S_0(f)\|_2
$$
and one shows that its kernel satisfies H\"{o}rmander's condition.

\begin{gracies}
 We are grateful to Piotr Hajlasz, who kindly encouraged us to publish
 the paper, and to Richard Wheeden for a useful correspondence
 concerning \cite{W1} and \cite{W2}.
 This work has been partially supported by grants 2009SGR420
(Generalitat de Catalunya) and  MTM2010-15657 (Spanish Ministry of
Science). The first named author acknowledges the generous support
received from a fellowship from the Spanish Ministry of Science and
from the Ferran Sunyer Balaguer Foundation. He is thankful to
Professor Hajlasz for his hospitality at Pittsburg University and
for many enlightening conversations.
\end{gracies}

\begin{tabular}{l}
Roc Alabern\\
Departament de Matem\`{a}tiques\\
Universitat Aut\`{o}noma de Barcelona\\
08193 Bellaterra, Barcelona, Catalonia\\
{\it E-mail:} {\tt roc.alabern@gmail.com}\\ \\
Joan Mateu\\
Departament de Matem\`{a}tiques\\
Universitat Aut\`{o}noma de Barcelona\\
08193 Bellaterra, Barcelona, Catalonia\\
{\it E-mail:} {\tt mateu@mat.uab.cat}\\ \\
Joan Verdera\\
Departament de Matem\`{a}tiques\\
Universitat Aut\`{o}noma de Barcelona\\
08193 Bellaterra, Barcelona, Catalonia\\
{\it E-mail:} {\tt jvm@mat.uab.cat}
\end{tabular}

\end{document}